\documentclass[final,hidelinks,onefignum,onetabnum]{siamart220329}

\usepackage{lipsum}
\usepackage{amsfonts}
\usepackage{graphicx}
\usepackage{epstopdf}
\usepackage{subcaption}
\usepackage{algorithmic}
\usepackage{fancyhdr}
\ifpdf
\DeclareGraphicsExtensions{.eps,.pdf,.png,.jpg}
\else
\DeclareGraphicsExtensions{.eps}
\fi

\usepackage{hyperref} 
\usepackage{cleveref}
\usepackage{booktabs}
\usepackage{stmaryrd}
\usepackage{bm}
\usepackage{caption}
\usepackage{multirow}
\usepackage{enumitem}

\usepackage{amsopn}

\newsiamthm{test}{Test}
\newsiamthm{assumption}{Assumption}
\newcommand{\nm}[1]{\left\|#1\right\|}

\def\jump#1{\llbracket #1 \rrbracket }

\newcommand{\ip}[2]{\left(#1,#2\right)}

\newcommand{\proj}{\widetilde{P}}
\newcommand{\ipproj}[2]{\mathbb{P}_{h}^{*}\left(#1,#2\right)}

\newcommand{\Pe}{\widetilde P}

\newsiamremark{remark}{Remark}
\newsiamremark{hypothesis}{Hypothesis}
\crefname{hypothesis}{Hypothesis}{Hypotheses}
\newsiamthm{claim}{Claim}
\newsiamremark{exmp}{Example} 

\headers{}{}

\title{Runge--Kutta-Aligned Oscillation Elimination: Restoring Superconvergence for Fully Discrete Shock-Capturing DG Schemes\thanks{This work was partially supported by Science Challenge Project (No.~TZ2025007), the Shenzhen Science and Technology Program (Nos.~JCYJ20250604144300001 and RCJC20221008092757098),	 the National Natural Science Foundation of China (Nos.~124B2022 and 12171227), and the Guangdong Basic and Applied Basic Research Foundation (2024A1515012329).}}
	
	\author{ 
		Manting Peng\thanks{Department of Mathematics, Southern University of Science and Technology, Shenzhen, Guangdong 518055, China (\email{pengmt2024@mail.sustech.edu.cn}).}
		\and Zhuoyun Li \thanks{Department of Mathematics, Southern University of Science and Technology, Shenzhen, Guangdong 518055, China (\email{lizhuoyun2022@mail.sustech.edu.cn}).  }
		\and Kailiang Wu\thanks{Corresponding author. Department of Mathematics and Shenzhen International Center for Mathematics, Southern University of Science and Technology, Shenzhen, Guangdong 518055, China (\email{wukl@sustech.edu.cn}).}} 
	
\ifpdf
\hypersetup{
	pdftitle={},
	pdfauthor={}
}
\fi

\allowdisplaybreaks

\newlist{steps}{enumerate}{1}
\setlist[steps, 1]{label =\textbf{ Step \arabic*:}}

\usepackage{multirow}

\begin{document}

	\maketitle
	\setlength{\abovedisplayskip}{4pt}   
	\setlength{\belowdisplayskip}{4pt}
	\setlength{\abovedisplayshortskip}{6pt}
	\setlength{\belowdisplayshortskip}{6pt}



	\begin{abstract}
Nonlinear stabilization is indispensable for discontinuous Galerkin (DG) discretizations of hyperbolic conservation laws, yet it typically disrupts the delicate error structure required for superconvergence analysis. Consequently, existing theory has largely been restricted to linear or semi-discrete schemes lacking oscillation control. This paper bridges this gap by proposing a Runge--Kutta (RK) aligned oscillation-eliminating (OE) DG framework that restores the superconvergence properties. By synchronizing the pseudo-time step in the OE procedure with the cumulative RK stage coefficients, we unlock a cancellation mechanism for low-order interface errors that is inaccessible to standard OEDG formulations. We rigorously prove that this fully discrete scheme achieves $(k+2)$-th order superconvergence to a tailored projection of the exact solution for linear conservation laws in both one and two dimensions, while maintaining the non-oscillatory shock-capturing capabilities of the original method. Moreover, we establish a general guiding principle for designing a class of OE-type DG schemes that exhibit such superconvergence. Key theoretical innovations include the construction of stage-aligned correction functions to compensate for nonlinear OE sources and the discovery of a two-dimensional projection operator that preserves outflow-edge averages, a property essential to close the discrete shift estimates. Numerical experiments confirm the predicted superconvergence rates and demonstrate that RK alignment preserves the parameter-free robustness of OEDG for problems with strong discontinuities.
	\end{abstract}

	\begin{keywords}
		discontinuous Galerkin methods, hyperbolic conservation laws, superconvergence, oscillation elimination, Runge-Kutta methods, shock-capturing
	\end{keywords}
	
	\begin{AMS}
		65M60, 65M12, 65M15, 35L65
	\end{AMS}

	\section{Introduction}
	Discontinuous Galerkin (DG) methods, originally introduced by Reed and Hill~\cite{reed1973triangular} and subsequently matured into the Runge--Kutta Discontinuous Galerkin (RKDG) framework by Cockburn, Shu, and collaborators~\cite{T_Cockburn_1989,T_Cockburn_1990,T_Cockburn_1998,shu2009discontinuous,O_Cockburn_2008}, have established themselves as a standard high-order methodology for solving hyperbolic conservation laws. A defining characteristic of DG schemes in the linear regime is their remarkable error structure: while the global $L^2$-error typically converges at order $k{+}1$ for degree-$k$ polynomials, specific observables---such as projections, cell averages, and boundary traces---exhibit \textit{superconvergence} at substantially higher rates. This phenomenon is far more than a theoretical curiosity; it constitutes the analytical backbone for high-order post-processing, asymptotically exact a posteriori error estimation, and efficient adaptive mesh strategies~\cite{ji2014superconvergent,E_Zhang_2004,adjerid2002posteriori,adjerid2006superconvergence,cao2014superconvergence,cao2017superconvergenceupwindbiased,xu2020superconvergence,xu2022superconvergence,A_Huang_2020,liu2020optimal,A_Yu_2020}.

However, a fundamental tension persists in the development of high-order schemes. Practical solvers for conservation laws must incorporate nonlinear stabilization mechanisms (such as limiters~\cite{shu2009discontinuous,qiu2005runge}, artificial viscosity~\cite{A_Yu_2020,A_Huang_2020}, or other methods~\cite{lin2024high,prill2009smoothed,hennemann2021provably}) to suppress spurious oscillations near discontinuities. Crucially, this nonlinearity typically disrupts the delicate error cancellation properties that superconvergence proofs rely upon. Consequently, a significant gap has widened in the literature: rigorous superconvergence analysis has been restricted almost exclusively to {\bf linear discretization} \cite{cao2014superconvergence,xu2020superconvergence,S_Cao_2015,cao2017superconvergenceupwindbiased,cheng2010superconvergence} schemes lacking oscillation control, while robust shock-capturing schemes have largely eluded sharp theoretical error characterization.
	
Recently, the Oscillation-Eliminating DG (OEDG) method~\cite{O_Peng_2024}, inspired by the oscillation-free damping technique of Lu, Liu, and Shu~\cite{lu2021oscillation}, has emerged as a promising alternative. By applying a scale- and evolution-invariant oscillation-eliminating (OE) procedure after each RK stage, OEDG ensures local conservation and robustness without problem-specific parameter tuning. While the method has demonstrated superior numerical performance across unstructured grids~\cite{ding2025robust} and hyperbolic systems~\cite{yan2024uniformly,cao2025robust,liu2025structure}, its theoretical foundation lags behind. In previous work~\cite{O_Peng_2024}, optimal convergence rates were proven via an $L^2$-stability argument, but the inherent nonlinearity of the stage-coupled OE damping term precluded a deeper analysis of the error structure. Specifically, a structural mismatch between the pseudo-time step of the OE procedure and the physical time steps of the Runge--Kutta stages destroys the interface error cancellations necessary for superconvergence, leaving the empirically observed high-order accuracy theoretically unexplained.

This paper bridges this gap by proposing a \emph{Runge--Kutta (RK) aligned} OE-type DG framework. {To the best of our knowledge, this represents the first fully discrete superconvergence theory for nonlinear shock-capturing DG schemes with oscillation control.} The central innovation is a structural synchronization: we modify the OE relaxation step so that its pseudo-time increment aligns precisely with the cumulative RK stage weights (Butcher coefficients). While deceptively simple implementation-wise, this alignment is profound analytically; it unlocks a stage-by-stage cancellation mechanism for low-order interface errors that is inaccessible in the original OEDG formulation. This allows us to construct novel \emph{stage-aligned correction functions} that systematically compensate for the nonlinear OE sources, thereby restoring the superconvergence property at the fully discrete level.

Extending this analysis to multidimensional and variable-coefficient problems introduces further technical obstacles, particularly the loss of Galerkin orthogonality in 2D and the errors introduced by ``freezing'' wind directions. To overcome these, we identify a specific 2D projection operator that preserves outflow-edge averages---a structure-preserving property essential for closing the discrete shift estimates.
	
Accordingly, the main contributions of this work are as follows:
\begin{itemize}[leftmargin=*, noitemsep]
	\item \textit{An RK-aligned OEDG framework:} We formulate a general class of OE-type schemes where the stabilization step is structurally coupled to the time integrator. This modification is non-intrusive, preserving the spatial discretization and shock-capturing robustness of the original OEDG while enabling rigorous analysis.
	
	\item \textit{Rigorous superconvergence theory:} We prove $(k{+}2)$-th order superconvergence of the fully discrete solution toward a tailored projection of the exact solution for linear hyperbolic equations. This result holds for both constant and variable coefficients in 1D and 2D. We further establish $(k{+}2)$-th order accuracy for physically relevant observables, including cell averages and outflow-edge averages.
	
	\item \textit{Novel analytical tools:} We develop an analysis framework for nonlinear schemes based on two key innovations: (1) stage-aligned correction functions that eliminate low-order nonlinear errors, and (2) a structure-preserving 2D projection that recovers approximate Galerkin orthogonality for variable-coefficient problems.
	
	\item \textit{Validation of robustness:} Numerical experiments confirm that the theoretical superconvergence rates are achieved in smooth regions, while the parameter-free shock-capturing capability of the original OEDG method is fully preserved in the presence of strong discontinuities.
\end{itemize}

This paper is organized as follows. In \Cref{sec:framework}, we present the RK-aligned OE-type DG framework and the structural assumptions on the OE coefficients. \Cref{sec:main_results} states the main superconvergence results. The analysis for constant-coefficient advection is provided in \Cref{sec:proof_const}.  Section~\ref{sec:VC} extends the analysis to linear variable-coefficient problems. Numerical experiments are reported in \Cref{sec:numerics}. Concluding remarks are given in \Cref{sec:conclusion}.

	\section{RK-aligned OE-type DG framework for hyperbolic conservation laws}\label{sec:framework}
	
	In this section, we propose an RK-aligned OE-type DG framework for hyperbolic conservation laws
	\begin{equation}\label{eq:HCL}
		\frac{\partial  \bm{u} }{\partial t}	 + \sum_{i = 1}^d \frac{\partial  \bm{f}_i(\bm{u})}{\partial x_i}  = \bm{0}, \qquad ({\bm x},t) \in \Omega \times \mathbb{R}^+,
	\end{equation}
	where $\bm{u} = (u_1,\ldots,u_N)^\top$ is the vector of conserved variables, $\mathbf{f}_i(\bm{u})$ denotes the flux in the $x_i$-direction and $\Omega$ represents the computational domain in $\mathbb{R}^d$. 
	
	\subsection{Formulation}
Let \(\mathcal{T}_h\) be a partition of the computational domain \(\Omega\). The DG finite element space associated with \(\mathcal{T}_h\) is defined as
\[
\mathbb{V}_h^k := \left\{ \bm{v} \in [L^2(\Omega)]^N : \bm{v}|_K \in [\mathbb{P}^k(K)]^N, \quad \forall K \in \mathcal{T}_h \right\}.
\]
We denote the mesh size by \( h := \max_{K\in\mathcal{T}_h}\mathrm{diam}(K) \).
For \eqref{eq:HCL}, the semi-discrete DG method seeks $\bm{u}_h(\cdot,t) \in \mathbb{V}_h^k$ such that
\begin{equation}\label{eq:semi-DG}
	\frac{\mathrm{d}}{\mathrm{d}t}\int_K\bm{u}_h \cdot \bm{v} \, \mathrm{d}\bm{x} =H(\bm{u}_h,\bm{v}),\quad H(\bm{u}_h,\bm{v}):= \int_K \bm{f}(\bm{u}_h) : \nabla \bm{v}\,\mathrm{d}\bm{x}
	- \sum_{e \subset \partial K} \int_e \hat{\bm{f}}(\bm{u}_h) : (\bm{n}_e \otimes \bm{v})\,\mathrm{d}S
\end{equation}
for all $\bm{v} \in \mathbb{V}_h^k$ and $K \in \mathcal{T}_h$. Here ``$:$'' denotes the Frobenius inner product of two matrices, ``$\otimes$'' represents the Kronecker product of two vectors, $\widehat {\bm f} ({\bm u}_h)$ denotes a suitable numerical flux on the element interface $e \in \partial K$, and ${\bm n}_e$ is the outward unit normal vector at $e$ with respect to $K$. Semi-discrete DG scheme \eqref{eq:semi-DG} can be expressed as a system of ordinary differential equations $\frac{\mathrm{d}}{\mathrm{d}t}\bm{u}_h = L_{\bm{f}}(\bm{u}_h)$. 
With an $r$th-order $s$-stage RK method, we can further discretize \eqref{eq:semi-DG} with respect to $t$ as the following fully-discrete RKDG scheme
\begin{subequations}\label{eq:RKDG}
	\begin{align}
		\bm{u}_h^{n,0} &= \bm{u}_h^n,\\
		\bm{u}_h^{n,l+1} &= \sum_{m=0}^{l}\big(c_{lm}\bm{u}_h^{n,m}+\tau d_{lm} L_{\bm{f}}(\bm{u}_h^{n,m}) \big),\quad l= 0,\dots,s-1,\label{eq:RKDG2}\\
		\bm{u}_h^{n+1} & = \bm{u}_h^{n,s},
	\end{align}
\end{subequations}
where $\bm{u}_h^n$ and $\tau>0$ denote the numerical solution at the $n$th time step and the time step size respectively. Throughout the paper, we consider the RK methods that satisfy $\sum_{m=0}^l c_{lm} = 1$, and  $\sum_{m=0}^l d_{lm} \geq0$ in \eqref{eq:RKDG2}, a class that includes, but is not limited to, the SSP RK methods \cite{gottlieb1998total}.

Due to the {nonlinear} nature, conventional {high-order} RKDG schemes often produce spurious oscillations near discontinuities. To mitigate these oscillations, \cite{O_Peng_2024} presents the OEDG method, which incorporates a novel oscillation-eliminating (OE) procedure after each RK stage of \eqref{eq:RKDG}. The resulting OEDG method not only captures sharp features effectively without spurious oscillations, but also exhibits several desirable properties such as scale-invariance and evolution-invariance. For $\bm{u}_h =(u_{h,1},\dots,u_{h,N}) \in \mathbb{V}_h^k$ and OE stepsize $\hat{\tau}>0$, the OE procedure is defined as $$\left( \mathcal{F}_{\hat{\tau}}\bm{u}_h \right) (\bm{x}) = \bm{u}_{\sigma}(\bm{x},\hat{\tau}), $$ where $\bm{u}_{\sigma}( \bm{x},\hat{t} )$ satisfies the following initial value problem: \begin{equation}\label{eq:OEprocedure}
	\begin{cases}
		\displaystyle \frac{\mathrm{d}}{\mathrm{d}\hat{t}} \int_K  \bm{u}_\sigma \cdot \bm{v} \, \mathrm{d}\bm{x} = -\sum_{p=0}^{k} \delta_K^p(\bm{u}_h) \int_K (\bm{u}_\sigma - { P^{(p-1)}} \bm{u}_\sigma) \cdot \bm{v} \, \mathrm{d}\bm{x}, \\
		\bm{u}_\sigma \big|_{\hat{t}=0} = \bm{u}_h,
	\end{cases}
\end{equation}
for all $K \in \mathcal{T}_h$ and $\bm{v} \in \mathbb{V}_h^k$. Here $\hat t$ is a pseudo-time variable, distinct from the physical time $t$. The operator $P^{(p-1)}$ is the
standard $L^2$-projection into $\mathbb V_h^{p-1}$ for $p\ge 1$, and we set $P^{(-1)}:=P^{(0)}$. The nonnegative coefficients $\{\delta_K^p(u_h)\}_{p=0}^k$ have a variety of choices. We require the OE coefficients satisfy the following general principle. 

\subsection{General guiding principle on the design of OE coefficients}\label{subsec:coe_preliminary}

\paragraph{\textbf{Condition 1} (Local Lipschitz-type continuity)}
For each \(p\), we assume that \(\delta_K^p(\cdot)\) satisfies a local Lipschitz-type continuity estimate. Specifically, if \(\|v-w\|_{L^{\infty}(\Omega)}\le Ch\), then
\begin{equation}\label{eq:lipschitz-typecontinuity}
	\sum_{K\in \mathcal{T}_h}\bigl(\delta_K^p(v)-\delta_K^p(w)\bigr)^2
	\le C\Bigl(|v|_{\mathrm{jump}}^2+|w|_{\mathrm{jump}}^2
	+ h^{-2}\,|v-w|_{\mathrm{jump}}^2 \Bigr).
\end{equation}
This condition is a stability bound expressed in terms of the jump seminorm across element interfaces, defined by
$
|\omega|_{\mathrm{jump}}^2 := \sum_{K\in\mathcal{T}_h}\sum_{e\subset\partial K}\int_{e}\jump{\omega}_e^{\,2}\,{\rm d}S,
$ 
where \(\jump{\omega}_e\) denotes the jump of \(\omega\) across the interface \(e\).
\begin{remark}
	In practice, verifying \(\|v-w\|_{L^{\infty}(\Omega)}\le Ch\) stage by stage may be inconvenient in the fully discrete analysis. As an alternative sufficient condition, we may assume the following scale-explicit estimate
	\begin{equation}\label{eq:lipschitz-typecontinuity2}
		\sum_{K\in \mathcal{T}_h}\bigl(\delta_K^p(v)-\delta_K^p(w)\bigr)^2
		\le C\,\frac{\|v-w\|_{L^{\infty}(\Omega)}^2}{h^2}\Bigl(|v|_{\mathrm{jump}}^2+|w|_{\mathrm{jump}}^2\Bigr)
		+ C\,h^{-2}\,|v-w|_{\mathrm{jump}}^2.
	\end{equation}
	Under the smallness condition \(\|v-w\|_{L^{\infty}(\Omega)}\le Ch\), \eqref{eq:lipschitz-typecontinuity2} immediately implies \eqref{eq:lipschitz-typecontinuity}. Furthermore, the $L^{\infty}$ norm can be substituted into other global norms such as the $L^p$ norm. In the fully discrete analysis in \Cref{sec:sup_analysis}, we will work with the scale-explicit estimate~\eqref{eq:lipschitz-typecontinuity2}.
\end{remark}

\paragraph{\textbf{Condition 2} (Controlled by jumps)}
For all $p$, we assume the damping coefficients $\{\delta_K^p(w)\}_{K\in\mathcal T_h}$ are nonnegative and satisfy
\begin{equation}\label{eq:jumpcontroll}
	\sum_{K\in\mathcal{T}_h}\Bigl(\delta_K^p(w)\Bigr)^2
	\le C\,h^{-2}\,|w|_{\mathrm{jump}}^2.
\end{equation}

\begin{remark}
A concrete choice of OE coefficients satisfying Conditions~1--2 is given in \cite[\S2]{O_Peng_2024}, where, for nonconstant initial value, Condition~2 is verified iteratively by showing that \(\|u_{\sigma}-|\Omega|^{-1}\int_\Omega u_{\sigma}\,dx\|\ge C\). Throughout this paper, we use only the abstract bounds stated in Conditions~1--2.
\end{remark}

Now we present the RK-aligned OE-type DG, which is a variant of the original OEDG method. For general hyperbolic system \eqref{eq:HCL}, the RK-aligned OE-type DG scheme can be written as
\begin{subequations}\label{eq:sd-OEDG}
	\begin{align}
		\bm{u}_\sigma^{n,0} &= \bm{u}_h^n,\\
		\bm{u}_h^{n,l+1} &= \sum_{m=0}^{l}\big(c_{lm}\bm{u}_\sigma^{n,m}+\tau d_{lm} L_{\bm{f}}(\bm{u}_\sigma^{n,m}) \big),
		\label{eq:sd-OEDG2}\\
		\bm{u}_\sigma^{n,l+1} & = \mathcal{F}_{\hat\tau_{l+1}}\bm{u}_h^{n,l+1}, \quad l= 0,\dots,s-1,	\label{eq:sd-OEDG3}\\
		\bm{u}_h^{n+1} & = \bm{u}_\sigma^{n,s}.
	\end{align}
\end{subequations}
At $(l+1)$th RK stage, we define the OE stepsize of the RK-aligned OE-type DG scheme as $\hat\tau_{l+1} =\sum_{m=0}^{l}d_{lm} \tau$. We also define \(\bm{u}_{\sigma}^n := \bm{u}_{\sigma}^{n,0}\) for all \(n\).

\begin{remark}[Why RK alignment matters]
	The OE relaxation~\eqref{eq:OEprocedure} acts as a nonlinear, mode-dependent damping in a pseudo-time variable.
	In a multi-stage RK update, the DG transport increment at stage $l{+}1$ is weighted by the coefficients $d_{lm}$.
	Choosing $\hat\tau_{l+1}=\tau\sum_{m=0}^{l}d_{lm}$ couples the damping strength to the \emph{same} stage weight.
	This alignment is the structural ingredient that restores the cancellation of low-order interface errors in the
	fully discrete error equation and is therefore indispensable for the superconvergence analysis in
	\Cref{sec:proof_const} and  \Cref{sec:VC}.
\end{remark}

	\section{Main results for linear advection with upwind fluxes}\label{sec:main_results}
\subsection{Unified notation and setting for theoretical analysis}

We consider the linear advection equation with periodic boundary conditions
\begin{equation}\label{eq:advec}
	u_t + \nabla\cdot ( \bm{\beta}u ) = 0,
\end{equation}
where \(\bm{\beta}\) denotes the velocity vector, which may be constant or {\it spatially dependent \(\bm{\beta}({\bm x})\)}. For our analysis, we employ the DG scheme equipped with upwind fluxes. We assume a bounded fixed wind direction and that without loss of generality, the velocity field $\bm{\beta}(\bm x)$ satisfies \( min_{1\le i\le d}\beta_i(\bm x) > 0\) for all ${\bm x}$ and \(\|\nabla \cdot \bm{\beta}\|_{L^\infty(\Omega)}\leq C\). To streamline the presentation, we introduce a unified set of notations for both one-dimensional (1D) and two-dimensional (2D) cases.

\begin{itemize}[leftmargin=*, noitemsep]
	\item \textbf{One dimension.}
	Let \(\mathcal{T}_h=\bigcup_i\{I_i=[x_{i-\frac12},x_{i+\frac12}]\}\) be a quasi-uniform partition of \(\Omega\). Quasi-uniformity implies the existence of a constant \(C> 0\) such that
	\(\min_i (x_{i+\frac12}-x_{i-\frac12}) \ge Ch\), where \(h = \max_i (x_{i+\frac12}-x_{i-\frac12})\).
	We define the local (cell-wise) inner product and norm as
	$
	\ip{\omega}{v}_i = \int_{I_i}\omega v\,\mathrm{d}x$ and 
	$
	\nm{\omega}_i = (\int_{I_i}\omega^2\,\mathrm{d}x)^{1/2}.
	$ 
	The global inner product and norm are obtained by summing over all cells:
	$
	\ip{\omega}{v} = \sum_i \ip{\omega}{v}_i$ and 
	$
	\nm{\omega} = (\sum_i \nm{\omega}_i^2)^{1/2}.
	$ 
	The DG spatial operator in \eqref{eq:semi-DG} satisfies \(H(\omega,v)=\sum_i H_i(\omega,v)\), where the local operator is given by
	\[
	H_i(\omega,v)
	= \ip{\beta\,\omega}{v_x}_i
	+ \beta\,\omega(x_{i-\frac12}^-)\,v(x_{i-\frac12}^+)
	- \beta\,\omega(x_{i+\frac12}^-)\,v(x_{i+\frac12}^-).
	\]
	
	\item \textbf{Two dimensions.}
	Let $\mathcal{T}_h=\bigcup_{i,j}\left\{
	I_{i,j}=[x_{i-\frac12},x_{i+\frac12}]\times[y_{j-\frac12},y_{j+\frac12}]
	\right\}$ be a uniform partition with constant step sizes \(h_x\) and \(h_y\) with $h = \sqrt{h_x^2+h_y^2}$.
	The local inner product and \(L^2\) norm are defined as
	$ 
	\ip{\omega}{v}_{i,j} = \int_{I_{i,j}}\omega\cdot v\,\mathrm{d}x\mathrm{d}y,
	$ and $
	\nm{\omega}_{i,j} =\ip{\omega}{\omega}_{i,j}^{1/2}.
	$ 
	Globally, we have
	$
	\ip{\omega}{v} = \sum_{i,j}\ip{\omega}{v}_{i,j}$ and 
	$ 
	\nm{\omega} = (\sum_{i,j}\nm{\omega}_{i,j}^2)^{1/2}.
	$ 
	The global 2D DG spatial operator is defined by summing the local contributions, \(H(\omega,v)=\sum_{i,j} H_{i,j}(\omega,v)\), where
	\[\resizebox{0.99\hsize}{!}{$
	\begin{aligned}
		H_{i,j}(\omega,v)
		&= \ip{\bm{\beta}\omega}{\nabla v}_{i,j}  +\int_{y_{j-\frac12}}^{y_{j+\frac12}}
		\Bigl[
		(\beta_1\omega)(x_{i-\frac12}^-,y)\,v(x_{i-\frac12}^+,y)
		-(\beta_1\omega)(x_{i+\frac12}^-,y)\,v(x_{i+\frac12}^-,y)
		\Bigr]\,\mathrm{d}y \\
		&\quad +\int_{x_{i-\frac12}}^{x_{i+\frac12}}
		\Bigl[
		(\beta_2\omega)(x,y_{j-\frac12}^-)\,v(x,y_{j-\frac12}^+)
		-(\beta_2\omega)(x,y_{j+\frac12}^-)\,v(x,y_{j+\frac12}^-)
		\Bigr]\,\mathrm{d}x.
	\end{aligned}$}
	\]
	This bilinear form \(H(\omega,v)\) represents the standard DG weak formulation of the advection operator \(-\nabla\cdot(\bm{\beta} \omega)\) with upwind numerical fluxes.
\end{itemize}

The RK-aligned OE-type DG method \eqref{eq:sd-OEDG2}--\eqref{eq:sd-OEDG3} can be reformulated as

	\begin{align}
		\ip{u_\sigma^{n,l+1}}{v}
		&= \sum_{m=0}^l\Bigl(
		c_{lm}\ip{u_\sigma^{n,m}}{v}
		+\tau d_{lm} H\!\left(u_\sigma^{n,m},v\right)
		\Bigr)
		+\tau \ip{\frac{\mathcal{F}_{\hat\tau_{l+1}}u_h^{n,l+1}-u_h^{n,l+1}}{\tau}}{v}
		\label{eq:sd-RKOEDG-advec}	.
	\end{align}


The term
\(
\bigl(\mathcal{F}_{\hat\tau_{l+1}}u_h^{n,l+1}-u_h^{n,l+1}\bigr)/\tau
\)
appears as a source contribution in \eqref{eq:sd-RKOEDG-advec}.
We refer to \eqref{eq:sd-RKOEDG-advec} as the RKDG reformulation of \eqref{eq:sd-OEDG} for \eqref{eq:advec}.

At each stage $m$, $u_h^{n,m}$ denotes the stage solution before applying the OE procedure, while $u_\sigma^{n,m}$ denotes the stage solution after the OE step.
In particular, $u_\sigma^{n,s}=u_\sigma^{n+1}$ is the fully updated solution at time $t^{n+1}$.
Throughout the paper, when the same superscript convention is used for other quantities, the index pair $(n,m)$ denotes the time level $t^n$ and the Runge--Kutta stage $m$ (with $n$ identified with $(n,0)$), and we define \(N := \lfloor T/\tau \rfloor\).

\subsection{Main results}
Our analysis relies on three main ingredients: a special projection $\Pe$, 
RK stage reference solutions $\{U^{n,m}\}$, and stage correction functions $\{Z^{n,m}\}$. 
We also use fully discrete $L^2$ stability of the underlying RKDG scheme with source terms (Lemma~\ref{assump:stability};  the variable-coefficient counterpart is proved in \Cref{sec:VC}); see, e.g., \cite{xu2020superconvergence,sun2019strong,xu2020error,xu2019l2}. We recall the stability estimate at the beginning of \Cref{sec:proof_const}.

\begin{proposition}[Optimal error estimate]\label{prop:opterror}
	Let $k\geq \frac d2$ and $r\geq \frac d2+1$. 
	Assume that the damping coefficients \(\{\delta_K^p\}\) satisfy \eqref{eq:lipschitz-typecontinuity2}--\eqref{eq:jumpcontroll}, and that the spatial mesh is quasi-uniform in one dimension and uniform Cartesian in two dimensions.
	Suppose that the exact solution $U$ of \eqref{eq:advec} is sufficiently smooth; for instance,
	$U \in C^{r+1}([0,T];H^{k+1}(\Omega))\text{ for }d=1$ and $
	U\in C^{r+1}\!\big([0,T];H^{k+2}(\Omega)\big)\text{ for }d=2.$
	If the initial error satisfies \(\|u_{\sigma}^0 - U(\cdot,0)\| \leq C h^{k+1}\), then under the CFL condition ${\frac{\tau}{h^\kappa} \leq C_{\textrm{CFL}}}$ in \Cref{assump:stability} (and in  \Cref{lem:VC:stab} for the variable-coefficient counterpart), there exist constants \(C>0\) and \(h_{\max}>0\), independent of \(h\) and \(\tau\), such that the RK-aligned OE-type RKDG solution \(u_{\sigma}^{n}\) satisfies
	\[
	\max_{0 \leq n \leq \lfloor T/\tau \rfloor} \|u_{\sigma}^n - U(\cdot,n\tau)\| \leq C\left(h^{k+1} + \tau^r\right),
	\]
	whenever $h \leq h_{\max}$.
\end{proposition}
	We also obtain optimal error estimates for shift differences (\Cref{prop:deltaerr}) in each coordinate direction, based on \eqref{eq:jumpcontroll} and \eqref{eq:lipschitz-typecontinuity}.
In particular, we define $\Delta_x w(x) := w(x+h_x)-w(x)$ for $d = 1$, and for $d=2$, 
\[
\Delta_x w(x,y) := w(x+h_x,y)-w(x,y),
\qquad
\Delta^y w(x,y) := w(x,y+h_y)-w(x,y).
\]
\Cref{prop:deltaerr} provides the optimal error bounds for $\Delta_x u_\sigma^n$ and $\Delta^y u_\sigma^n$.

\begin{proposition}[Optimal discrete shifting estimate]\label{prop:deltaerr}
	In addition to the assumptions of \Cref{prop:opterror}, suppose that
	$U \in C^{r+1}\!\bigl([0,T];H^{k+2}(\Omega)\bigr)$ in 1D, and
	$U \in C^{r+1}\!\bigl([0,T];H^{k+3}(\Omega)\bigr)$ in 2D.
	Assume further that the initial data for the RK-aligned OE-type DG scheme \eqref{eq:sd-OEDG} satisfy
\begin{equation*}
	\nm{\Delta_{x}u_{\sigma}^0-\Delta_{x}U(\cdot,0)} \le Ch^{k+2}\quad \text{in 1D},
\end{equation*}
and, in 2D,
\begin{equation*}
	\nm{\Delta_{x}u_{\sigma}^0-\Delta_{x}U(\cdot,0)} \le Ch^{k+2},
	\qquad
	\nm{\Delta^{y}u_{\sigma}^0-\Delta^{y}U(\cdot,0)} \le Ch^{k+2}.
\end{equation*}
Then, whenever $\tau/h^{\kappa} \le C_{\mathrm{CFL}}$ and $h \le h_{\max}$, we have
\begin{equation*}
	\max_{0 \le n \le \lfloor T/\tau \rfloor}
	\nm{\Delta_{x}u_{\sigma}^n-\Delta_{x}U(\cdot,n\tau)}
	\le C\left(h^{k+2} + \tau^r\right)
	\quad \text{in 1D},
\end{equation*}
and, in 2D,
\begin{align*}
&	\max_{0 \le n \le \lfloor T/\tau \rfloor}
	\nm{\Delta_{x}u_{\sigma}^n-\Delta_{x}U(\cdot,n\tau)}
	\le C\left(h^{k+2} + \tau^r\right),\\
&	\max_{0 \le n \le \lfloor T/\tau \rfloor}
	\nm{\Delta^{y}u_{\sigma}^n-\Delta^{y}U(\cdot,n\tau)}
	\le C\left(h^{k+2} + \tau^r\right).
\end{align*}
\end{proposition}

\begin{theorem}[Main superconvergence result]\label{thm:superconvergence}
	Under the assumptions of \Cref{prop:deltaerr}, suppose that the RK-aligned OE-type DG scheme is initialized by 
	$
	u_{\sigma}^{0}=\Pe U(\cdot,0),
	$ 
	where $\Pe$ is the right Gauss--Radau projection in 1D, see \Cref{def:1Dprojection}, and the projection defined in \Cref{def:2Dprojection} in 2D (for the variable-wind case,  See \Cref{def:2Dproj_linvar} for the frozen-direction variant). With this choice, the initial discrete shifting conditions required in \Cref{prop:deltaerr} are satisfied. Then
	\begin{equation}\label{equ:errorest}
		\max_{0\le n\le N}\|u_{\sigma}^{n}-\Pe U(\cdot,t_n)\|\le C\,(h^{k+2}+\tau^{r}).
	\end{equation}
	Moreover, for \eqref{eq:advec}, the errors in the cell averages ($e_1$) and in the outflow-edge averages ($e_2$) satisfy
	\begin{equation}\label{equ:superest}
		\max_{0\le n\le N} e_1^n \le C\,(h^{k+2}+\tau^{r}),
		\qquad
		\max_{0\le n\le N} e_2^n \le C\,(h^{k+2}+\tau^{r}).
	\end{equation}
	Here we define
	\begin{equation}\label{equ:error}	\begin{aligned}
			e_1^n &:= \left(\sum_{K\in {\mathcal T}_h} |K|
			\left(\frac{1}{|K|}\int_K \bigl(u_\sigma^n-U(\cdot,t_n)\bigr)\,\mathrm{d}\bm{x}\right)^2\right)^{1/2},
			\\
			e_2^n &:= \left(\sum_{K\in {\mathcal T}_h} |K|
			\left(\frac{1}{|\partial^{\mathrm{out}} K|}\int_{\partial^{\mathrm{out}} K} \bigl(u_{\sigma}^n-U(\cdot,t_n)\bigr)\,\mathrm{d}s\right)^2\right)^{1/2},\end{aligned}
	\end{equation}
	where $\partial^{\mathrm{out}} K$ denotes the outflow boundary of the cell $K$. In 1D, $e_2^n$ is understood as
$	e_2^n
	=\left(h\sum_{i} 
	\bigl(u_\sigma^n(x_{i+\frac12}^-)-U(x_{i+\frac12},t_n)\bigr)^2\right)^{1/2}.
$
\end{theorem}

Once \eqref{equ:errorest} is established, the estimates \eqref{equ:superest} follow immediately.
	Since the projection $\Pe U$ preserves both the cell average, see \Cref{def:1Dprojection,def:2Dprojection}, and the outflow-edge average, see \Cref{prop:structurepreserve_proj}, the errors $e_1^n$ and $e_2^n$ can be rewritten in terms of $\zeta_\sigma^n$ as
	\begin{align*}
	e_1^n
	&= \left(\sum_{K \in \mathcal{T}_h} |K|
	\left( \frac{1}{|K|} \int_K \zeta_{\sigma}^n \,\mathrm{d}\bm{x} \right)^2\right)^{1/2}
	\le \|\zeta_{\sigma}^n\|,
	\\
	e_2^n
	&= \left(\sum_{K \in \mathcal{T}_h} |K|
	\left( \frac{1}{|\partial^{\mathrm{out}} K|} \int_{\partial^{\mathrm{out}} K} \zeta_{\sigma}^n \,\mathrm{d}s \right)^2\right)^{1/2}
	\le \|\zeta_{\sigma}^n\|.
	\end{align*}
	The final inequalities follow from standard inverse estimates.

Throughout the proof of the main results, we will repeatedly use the following standard inverse estimates on $\forall v\in \mathbb{P}^k(K)$, where $C$ depends only on the polynomial degree and the mesh regularity
	\begin{align} \label{eq:inveq}
		\nm{\nabla v}_{L^2(K)} \le Ch^{-1}\nm{v}_{L^2(K)},  \quad 
		\nm{v}_{L^\infty(K)} \le Ch^{-\frac{d}{2}}\nm{v}_{L^2(K)}.  
	\end{align}
	
	\section{Proofs of main results for linear constant advection equations}\label{sec:proof_const}
	\subsection{\texorpdfstring{$L^2$}{L2} stability of the RKDG method}
	\label{sec:opt}
	As stated earlier, the fully discrete error analysis will be based on the following $L^2$ stability results for the standard RKDG method with sources, which follows from \cite{xu2020superconvergence,sun2019strong,xu2020error,xu2019l2} and was presented in  \cite[Proposition 4.4]{O_Peng_2024}. 
	\begin{lemma}[RKDG $L^2$–stability with sources]\label{assump:stability}
		Consider the RKDG scheme without the OE procedure but with stage source $\{g^{n,m}\}$ formulated as
		\begin{subequations}
			\begin{align*}
				u_h^{n,0} &= u_h^{n},\\
				\ip{u_h^{n,\ell+1}}{v} &= \sum_{m=0}^{\ell}\!\Big(c_{\ell m}\,\ip{u_h^{n,m}}{v}
				+ \tau\, d_{\ell m}\, H\!\left(u_h^{n,m},v\right)\Big)
				+ \tau\,\ip{g^{n,\ell+1}}{v},\quad \ell=0,\dots,s-1,\\
				u_h^{n+1} &= u_h^{n,s},
			\end{align*}
		\end{subequations}
		with test functions $v\in\mathbb{V}_h^{k}$ and stage data $\{g^{n,\ell+1}\}\subset\mathbb{V}_h^{k}$. Under a CFL constraint $\tau\le C_{\mathrm{CFL}} h^\kappa$, there exists $C>0$ independent of $h,\tau$ such that
		\begin{equation}\label{eq:stability}
			\|u_h^{n+1}\|^{2}\;\le\; \bigl(1+C_{\mathrm{s}}\tau\bigr)\,\|u_h^{n}\|^{2}
			\;+\; C\,\tau\sum_{\ell=0}^{s-1}\!\|g^{n,\ell+1}\|^{2}.
		\end{equation}
		Specifically, for an $r$–stage, $r$th-order explicit RK method one may choose $\kappa$ as in \cite{xu2020superconvergence,sun2019strong}.
	\end{lemma}

	\subsection{Auxiliary estimates for the OE stabilization}\label{sec:aux_opt}
The fully discrete superconvergence argument builds on two auxiliary ingredients inherited from the OE framework:
(i) an optimal-order error estimate for the stabilized solution and its discrete shifts, and 
(ii) quantitative bounds on the OE coefficients and on the OE procedure increment. 
	The RK-aligned OE procedure $\mathcal{F}_{\hat\tau_{l+1}}$ exhibits similar properties to those of the original OE procedure of the standard OEDG method \cite{O_Peng_2024}.  As this paper focuses on the superconvergence properties of the RK-aligned OE-type DG scheme \eqref{eq:sd-OEDG}, we omit the analogous proof details here. The proof of Proposition~\ref{prop:opterror} proceeds by invoking the error decomposition \eqref{eq:notation0} and following the arguments established in \cite{O_Peng_2024}. Specifically, by deriving the intermediate estimates in Lemmas~\ref{lmm:A2}--\ref{lmm:A3} and applying a discrete Gr\"onwall inequality, we verify the a priori assumption $\|\zeta^{n,l+1}\|_{L^{\infty}(\Omega)}\le h$ to close the mathematical induction. This yields the final error bound
		$$ \nm{\zeta^{n,l+1}} \leq C\bigl( h^{k+1} + \tau^r \bigr), $$
		which completes the proof of Proposition~\ref{prop:opterror}. Furthermore, Proposition~\ref{prop:opterror}, in conjunction with Lemmas~\ref{lmm:A2} and \ref{lmm:A4}, immediately implies Lemma~\ref{lmm:OEerr}.
		
		
		\begin{lemma}\label{lmm:A2}
			Using inverse inequalities, \eqref{eq:jumpcontroll} yields the following bound for the damping coefficients
			\begin{equation}
				\sum_{K \in \mathcal{T}_h}\left( \delta_K^p\left( u_h^{n,l+1} \right) \right)^2 \leq Ch^{-2-d}\left( \nm{\zeta^{n,l+1}}^2 + h^{2k+2}\nm{U^{n,l+1}}_{H^{k+1}\left( \Omega \right)}^2 \right).
			\end{equation}
		\end{lemma}
		
		\begin{lemma}\label{lmm:A3}
			If $\nm{\zeta^{n,l+1}}_{L^{\infty}\left( \Omega \right)} \leq h$, then for $0 \leq p \leq k$,
			\begin{equation}
				\nm{u_h^{n,l+1}-P^{(p-1)}u_h^{n,l+1}}_{L^2\left( K \right)} \leq Ch^{1+\frac{d}{2}}	
			\end{equation}
		\end{lemma}

		Based on the analysis of \cite{O_Peng_2024}, it suffices to verify the following lemma for the RK-aligned OE-type DG scheme \eqref{eq:sd-OEDG}.
		\begin{lemma}\label{lmm:A4}
			If $\nm{\zeta^{n,l+1}}_{L^{\infty}\left( \Omega \right)} \leq h \leq C_{\star,l+1}$ and $\frac{\tau}{h^{\kappa}}\leq C_{\text{CFL}}$, we have
			\begin{equation}
				\nm{\frac{\mathcal{F}_{\hat{\tau}^{l+1}}u_h^{n,l+1} - u_h^{n,l+1}}{\tau}} \leq C\left( \nm{\zeta^{n,l+1}} + h^{k+1} \right).
			\end{equation}
	\end{lemma}
	Based on the optimal error estimate (\Cref{prop:opterror}) and \eqref{eq:jumpcontroll}, we can derive the following estimate that will be useful in the superconvergence proof. 
	\begin{lemma}\label{lmm:OEerr}
		Under the constraint ${\frac{\tau}{h^{\kappa}} \leq C_{\text{CFL}}}$ and $h \leq h_{\max}$, 
		\begin{equation*}
			\sum_{K \in \mathcal{T}_h}\left( \delta_K^p\left( u_h^{n,m} \right) \right)^2 \leq Ch^{-2-d}\left( h^{2k+2} + \tau^{2r} \right),
		\end{equation*}
		for some constant $C>0$ independent of $m$ and $n$. Therefore,
		\begin{equation*}
			\nm{\mathcal{F}_{\hat\tau_{l+1}}u_h^{n,l+1} - u_h^{n,l+1}} \leq C\tau \left( h^{k+1} + \tau^r \right)\quad l=0,\dots,s-1,\ \ \forall n.
		\end{equation*}
	\end{lemma}

	Combining \Cref{prop:deltaerr} with \eqref{eq:jumpcontroll} and \eqref{eq:lipschitz-typecontinuity}, we can further derive estimates for the discrete shifts
	$\Delta_x\delta_K^p(\bm{u}_h)$ and $\Delta^y\delta_K^p(\bm{u}_h)$.
	
	Compared with the 1D case, the 2D proof is more technical, and therefore we only prove Proposition 3.2 in 2D in this subsection. At each RK stage, following \eqref{eq:notation0}, we decompose $\Delta_x u_{\sigma}^{n,m} - \Delta_x U^{n,m}$ into
		$
		\Delta_x u_{\sigma}^{n,m} - \Delta_xU^{n,m}=\Delta_x\zeta_{\sigma}^{n,m} - \Delta_x \eta^{n,m}.
		$
		The approximation of projection $\Pe $ yields that
		\begin{equation}\label{eq:0720-1}
			\nm{\Delta_x \eta^{n,m}}_{i,j} \leq Ch^{k+1}\nm{\Delta_x U^{n,m}}_{H^{k+1}\left( I_{i,j} \right)} \leq Ch^{k+3}\nm{U^{n,m}}_{W^{k+2,\infty}\left( \Omega \right)},
		\end{equation}
		so $\nm{\Delta_x \eta^{n,m}} \leq Ch^{k+2}$.
		To prove Proposition \ref{prop:deltaerr}, we now verify that
		$
		\nm{\Delta_x\zeta_{\sigma}^{n}} \leq C\left( h^{k+2} + \tau^r \right).
		$ From \eqref{eq:0630-1} and the definitions of $\ip{\cdot}{\cdot}$ and $H\left( \cdot,\cdot \right)$, we know that
		\begin{equation}
			\begin{aligned}
				\ip{\Delta_x\zeta_{\sigma}^{n,l+1}}{v} & = \sum_{m=0}^l \big(  c_{lm}\ip{\Delta_x\zeta_{\sigma}^{n,m}}{v} + \tau d_{lm} H\left( \Delta_x\zeta_{\sigma}^{n,m}, v \right) \big) \\
				& + \tau\ip{\Delta_x\mathcal{S}_1^{n,l+1}}{v} + \tau\ip{\Delta_x\mathcal{S}_2^{n,l+1}}{v}.
			\end{aligned}
		\end{equation}
		For $\Delta_x\mathcal{S}_1^{n,l+1}$, we can use estimation technique in \eqref{eq:0720-1} and derive that
		$
		\nm{\Delta_x\mathcal{S}_1^{n,l+1}} \leq C\left( h^{k+2} + \tau^r \right).
		$
		For $\Delta_x\mathcal{S}_2^{n,l+1}$, we have the following observation.
		$$\resizebox{0.99\hsize}{!}{$
		\begin{aligned}
			\nm{\Delta_x\mathcal{S}_2^{n,l+1}} & \leq C\left( \sum_{p=0}^k\sum_{i,j}\left( \delta_{I_{i+1,j}}^p\left( u_h^{n,l+1} \right) \right)^2\nm{\Delta_x u_h^{n,l+1}-P^{(p-1)}\Delta_x u_h^{n,l+1}}_{i,j}^2  \right)^\frac{1}{2} \\
			& + C\left( \sum_{p=0}^k\sum_{i,j}\left( \delta_{I_{i+1,j}}^p\left( u_h^{n,l+1} \right) - \delta_{I_{i,j}}^p\left( u_h^{n,l+1} \right) \right)^2\nm{ u_h^{n,l+1}-P^{(p-1)} u_h^{n,l+1}}_{i,j}^2 \right)^\frac{1}{2}.
		\end{aligned}$}$$
		Since we can obtain that
		
		\begin{align*}
			&\nm{\Delta_x u_h^{n,l+1}-P^{(p-1)}\Delta_x u_h^{n,l+1}}_{i,j} \\
			& \leq C\left( \nm{\Delta_x \zeta_h^{n,l+1}}_{i,j} + \nm{\Delta_x \eta^{n,l+1}}_{i,j} + \nm{\Delta_x U^{n,l+1}-P^{(0)}\Delta_x U^{n,l+1}}_{i,j}  \right) \\
			& \leq C \nm{\Delta_x \zeta_h^{n,l+1}}_{i,j} + Ch\nm{\Delta_x U^{n,l+1}}_{H^1\left( I_{i,j} \right)}  \leq C\nm{\Delta_x \zeta_h^{n,l+1}}_{i,j} +{ Ch^2\nm{ U^{n,l+1}}_{H^{3}\left( I_{i,j} \right)}},
		\end{align*}
		
		\begin{equation}\label{eq:0720-2}
			\resizebox{0.99\hsize}{!}{$	\begin{aligned}
				\sum_{i,j} \bigg( &\delta_{I_{i+1,j}}^p\left( u_h^{n,l+1} \right) - \delta_{I_{i,j}}^p\left( u_h^{n,l+1} \right) \bigg)^2\\& \leq Ch^{-2}\sum_{|{\bm \alpha}|=p}\sum_{i,j}\sum_{e\in\partial I_{i,j}}\int_e \jump{\partial^{\bm \alpha} \left( \Delta_xu_h^{n,l+1} \right)}_e^2\, {\rm d}S+ Ch^{2k}\nm{U^{n,l+1}}_{H^{k+2}\left( \Omega \right)}^2 \\
				& \leq Ch^{-2}\sum_{|{\bm \alpha}|=p}\sum_{i,j}\sum_{e\in\partial I_{i,j}} \int_e \left( \jump{\partial^{\bm \alpha} \left( \Delta_x \zeta_h^{n,l+1} \right)}_e^2 + \jump{\partial^{\bm \alpha}\left( \Delta_x \eta^{n,l+1} \right)}_e^2 \right) \, {\rm d}S+ Ch^{2k}\nm{U^{n,l+1}}_{H^{k+2}\left( \Omega \right)}^2 \\
				& \leq Ch^{-4}\nm{\Delta_x \zeta_h^{n,l+1}}^2 + Ch^{2k}\nm{U^{n,l+1}}_{H^{k+2}\left( \Omega \right)}^2,
			\end{aligned}$}
		\end{equation}
		and in the proof of \Cref{prop:opterror}, we have shown that $\nm{\zeta^{n,l+1}}_{L^{\infty}\left( \Omega \right)} \leq h$ when $h\leq h_{\max}$. By Lemma \ref{lmm:OEerr}, Lemma \ref{lmm:A2}, and \eqref{eq:0720-1},
		$
		\nm{\Delta_x\mathcal{S}_2^{n,l+1}} \leq C\left( \nm{\Delta_x \zeta_h^{n,l+1}} + h^{k+2} \right).
		$
		In the meantime, \eqref{eq:0630-1} also suggests that
		\begin{align*}
			\ip{\Delta_x\zeta_{h}^{n,l+1}}{v} = \sum_{m=0}^l \big(  c_{lm}\ip{\Delta_x\zeta_{\sigma}^{n,m}}{v} + \tau d_{lm} H\left( \Delta_x\zeta_{\sigma}^{n,m}, v \right) \big)  + \tau\ip{\Delta_x\mathcal{S}_1^{n,l+1}}{v}.
		\end{align*}
		Therefore, with the arguments in the proof of \Cref{prop:opterror} and induction hypothesis, we obtain
		\begin{equation}\label{eq:0720-3}
			\nm{\Delta_x\zeta_{h}^{n,l+1}} \leq C\nm{\Delta_x\zeta_{\sigma}^{n}} + C\left( h^{k+2} + \tau^r \right).
		\end{equation}
		Under Lemma~\ref{assump:stability}, this indicates that
		\begin{equation}
			\nm{\Delta_x\zeta_{\sigma}^{n+1}}= \nm{\Delta_x\zeta_{\sigma}^{n,s}} \leq \left( 1+C\tau \right)\nm{\Delta_x\zeta_{\sigma}^{n}} + C\tau\left( h^{k+2} + \tau^r \right).
		\end{equation}
		Finally, using the {Gr\"onwall} inequality, we prove that $\nm{\Delta_x\zeta_{\sigma}^{n}} \leq C\left( h^{k+2} + \tau^r \right)$. Thus,
		$
		\nm{\Delta_x u_{\sigma}^n-\Delta_x U^n} \leq \nm{\Delta_x\zeta_{\sigma}^{n}} + \nm{\Delta_x\eta^{n}} \leq C\left( h^{k+2} + \tau^r \right).
		$
		The above strategy can also be applied to prove that
		$
		\nm{\Delta^y u_{\sigma}^n-\Delta^y U^n} \leq C\left( h^{k+2} + \tau^r \right).
		$
		
		Having proven Proposition \ref{prop:deltaerr}, Lemma \ref{lmm:deltaOEerr-1} can be verified with \eqref{eq:0720-2} and \eqref{eq:0720-3}.

	Proceeding as in \Cref{lmm:OEerr}, we obtain the following estimate.
	\begin{lemma}\label{lmm:deltaOEerr-1}
		Under the assumptions of \Cref{prop:deltaerr} and the local Lipschitz-type continuity condition, there exists a constant $C>0$, independent of $n$, such that
		\begin{align*}
				&\sum_{i,j}\Bigl( \delta_{I_{i+1,j}}^p\bigl( u_h^{n,m} \bigr) - \delta_{I_{i,j}}^p\bigl( u_h^{n,m} \bigr) \Bigr)^2
				\le Ch^{-2-d}\left( h^{2k+4} + \tau^{2r} \right),\\
				&\sum_{i,j}\Bigl( \delta_{I_{i,j+1}}^p\bigl( u_h^{n,m} \bigr) - \delta_{I_{i,j}}^p\bigl( u_h^{n,m} \bigr) \Bigr)^2
				\le Ch^{-2-d}\left( h^{2k+4} + \tau^{2r} \right).
			\end{align*}
	\end{lemma}

	\subsection{Superconvergence analysis}\label{sec:superconvergence}
	In this subsection, we rigorously establish the superconvergence of the RK-aligned OE-type RKDG scheme \eqref{eq:sd-OEDG} for the linear advection equation \eqref{eq:advec} with constant velocity \(\bm{\beta}\). For notational simplicity, we take $\beta=(1,\dots,1)$ in the analysis below;
	the case of a general constant velocity with $\beta_i>0$ can be handled analogously.
		
	We first present an overview of the proof of \Cref{thm:superconvergence} and defer the precise definitions of the source error terms \(\left\{\mathcal{S}^{n,l+1} \right\}\), $\left\{ \widetilde{\mathcal{S}}^{n,l+1} \right\}$, and the correction functions \(\{Z^{n,m}\}\) to subsequent content.
	
	\begin{steps}
		\item {\bf Error decomposition.} At the $m$th RK stage of the $n$th time step, 
		decompose the error $u_{\sigma}^{n,m}-U^{n,m}$ 
		using the projection operator $\Pe $  as 
		\begin{equation}\label{eq:notation0}
			\zeta_{\sigma}^{n,m}=u_{\sigma}^{n,m}-\Pe U^{n,m},\ \ \zeta^{n,m}=u_{h}^{n,m}-\Pe U^{n,m},\ \ \eta^{n,m}=U^{n,m}-\Pe U^{n,m},
		\end{equation}
		where $\left\{ \zeta_{\sigma}^{n,m} \right\}$ satisfies an RKDG scheme with nonlinear source terms $\left\{ \mathcal{S}^{n,l+1} \right\}$. 
		
		\item {\bf Correction functions.} Define appropriate correction functions $\left\{ Z^{n,m} \right\}$ so that, in \eqref{eq:0630-1}, the low-order part of the stage-source error is canceled for all nonconstant test polynomials, while the constant mode vanishes by \Cref{prop:structurepreserve_proj}. We prove $\nm{Z^{n,m}} \leq C\left( h^{k+2} + \tau^r \right)$. 
		
		\item {\bf Source term correction.} A key property is that the projection no matter in 1D or 2D preserves its cell average and outflow edge average  (\Cref{prop:structurepreserve_proj}), which implies that the spatial error vanishes for piecewise constant test functions. Add $\left\{ Z^{n,m} \right\}$ into the RKDG scheme derived in \textbf{Step 1} to obtain an RKDG scheme for $\left\{ \widetilde{\zeta}_{\sigma}^{n,m}:=\zeta_{\sigma}^{n,m}+Z^{n,m} \right\}$ with corrected source terms $\left\{ \widetilde{\mathcal{S}}^{n,l+1} \right\}$.
		
		\item {\bf {Gr\"onwall} argument.} We show $\nm{\widetilde{\mathcal{S}}^{n,l+1}} \leq C\left( h^{k+2} + \tau^r \right)$. Then \Cref{assump:stability} yields that
		\begin{equation*}
			\nm{\widetilde{\zeta}_{\sigma}^{n+1}} \leq \left( 1+C\tau \right)\nm{\widetilde{\zeta}_{\sigma}^{n}} + C\tau\left( h^{k+2} + \tau^r \right)\quad \forall n.
		\end{equation*}
		Since $\nm{\widetilde{\zeta}_{\sigma}^{0}} = \nm{Z^0} \leq C\left( h^{k+2} + \tau^r \right)$ (by \textbf{Step 2}), we use the {Gr\"onwall}'s inequality and show that
		\begin{equation*}
			\nm{u_{\sigma}^n - \Pe U(\cdot,t_n)} \leq \nm{\widetilde{\zeta}_{\sigma}^{n}} + \nm{Z^n} \leq C\left( h^{k+2} + \tau^r \right).
		\end{equation*}
	\end{steps}
	\begin{remark}
		The superconvergence analysis in 1D and 2D appears to differ significantly. Since correction functions can only correct terms where the test function is not piecewise constant, we need to demonstrate that the scheme itself exhibits superconvergence. In 1D, the projection satisfies \(H_{i}( \Pe  (\omega),v ) = H_{i}( \omega,v )\). However, in 2D, unlike in 1D, it is possible that \(H_{i,j}( \Pe  (\omega),v ) \neq H_{i,j}\left( \omega,v \right)\). Nevertheless, a key observation is that the 2D projection preserves both its cell average and outflow edge averages, resulting in vanishing error in the spatial operator as described in \Cref{prop:structurepreserve_proj}.
	\end{remark}
	\begin{remark}
		Although our result in \Cref{thm:superconvergence} established in \Cref{sec:sup_analysis} focuses on the $\mathbb{P}^k$-based RK-aligned OE-type DG method,  
		superconvergence results for the $\mathbb{Q}^k$ elements can be easily derived within the same analytical framework.
	\end{remark}

	\subsubsection{Reference solutions}
	Following the roadmap outlined in Step 1 of our proof sketch, we begin by defining the reference solutions $\{U^{n,m}\}$ which serve as the high-order accurate targets for our analysis. 
	Following the approach in \cite{xu2020superconvergence}, we define a stage-by-stage reference solution $U^{n,m}$ based on the exact solution $U(x,t)$: 
	\begin{subequations}\label{eq:refsol}
		\begin{align}
			U^{n,0} & =  U(\cdot,t_n), \\
			U^{n,l+1} & = \sum_{m=0}^{l}\left( c_{lm}U^{n,m}-\tau d_{lm} {  \bm{\beta}\cdot \nabla U^{n,m} }\right),\quad l=0,\dots,s-2, \label{eq:0709-2} \\
			U^{n,s} & = U(\cdot,t_{n+1}).
		\end{align}
	\end{subequations}
	Denoting $U^n := U(\cdot, t_n)$ for each time step $t_n$, we can rewrite \eqref{eq:refsol}  
	in the formulation of the RKDG scheme as follows: 
	\begin{subequations}\label{eq:RKrefsol}
		\begin{align}
			U^{n,0} & = U^n  , \\
			\ip{U^{n,l+1}}{v} & = \sum_{m=0}^{l}\big( c_{lm}\ip{U^{n,m}}{v}+\tau d_{lm}H\left( U^{n,m},v \right) \big) + \tau\ip{\rho^{n,l+1}}{v}, \\
			U^{n+1} & = U^{n,s},
		\end{align}
	\end{subequations}
	where $\rho^{n,l+1}$ is the local truncation error defined as
	\begin{equation}\label{eq:lte}
		\rho^{n,l+1} = \begin{cases}
			\frac{U^{n+1}-\sum_{m=0}^{s-1}\left( c_{lm}U^{n,m}-\tau d_{lm} {  \bm{\beta}\cdot \nabla U^{n,m}} \right)}{\tau},\ \ & l=s-1,\\
			0,\ \ & otherwise.
		\end{cases}
	\end{equation}
	Since we assume that $U$ is sufficiently smooth, for an $r$th-order RK method, $\rho^{n,m}$ should satisfy that
	\begin{equation*}
		\nm{\rho^{n,l+1}}_{L^{\infty}(\Omega)} \leq C \tau^r \quad \forall0\leq l \leq s-1,\ \forall n.
	\end{equation*}
	
	\subsubsection{Projection operator} 
	Following the roadmap, we first define the projection operator $\Pe $ that sets the target for our superconvergence estimate. 
	For $d=1$ and $2$, the projection operator $\Pe $ has different definitions in our superconvergence analysis.
	
	\subsubsection*{\textbf{1D projection}}
	When $d=1$, $\Pe $ is chosen to be the right{ Gauss--Radau} projection defined in \Cref{def:1Dprojection}.
	\begin{definition}[Projection for 1D problem]\label{def:1Dprojection}
		Let $\ipproj{\omega}{v}_i=-\ip{\omega}{\frac{\mathrm{d}v}{\mathrm{d}x}}_i$. For $\omega \in H^1(\Omega)$, $\Pe  \omega$ is the DG polynomial in $\mathbb{V}_h^k$ such that
		\begin{equation}
			\begin{cases}
				\ipproj{\proj \omega}{v}_i=\ipproj{\omega}{v}_i,\ \ & \forall v \in \mathbb{P}^k(I_i),\ \ \forall i, \\
				\left( \Pe \omega \right) \left( x_{i+\frac{1}{2}}^- \right) = \omega\left( x_{i+\frac{1}{2}}^- \right),\ \ & \forall i.
			\end{cases}
		\end{equation}
		
	\end{definition}
	 The following lemma demonstrates the approximation property and the superconvergence property of the right Gauss--Radau projection.
	\begin{lemma}
		If $\omega \in H^{k+1}(\Omega)$ is sufficiently smooth, with Definition \ref{def:1Dprojection},
		\begin{align*}
				\nm{\omega - \Pe \omega}_i &\leq Ch^{m} \nm{\omega}_{H^{m}(I_i)}\quad m=1,\dots,k+1,\ \ \forall i, \\
			H_i(\omega-\widetilde P\omega,v)&=0,\qquad \forall v\in \mathbb{P}^k(I_i),\ \forall i.
		\end{align*}
	\end{lemma}
	Moreover, the right Gauss--Radau projection is $L^\infty$ stable, as shown in \cite[Lemma 3.4]{yang2012analysis}.  

	\subsubsection*{\textbf{2D projection}}
	The superconvergence analysis in 2D presents a significant challenge due to the loss of a simple Galerkin orthogonality property for the error. A major technical obstacle arises when estimating flux error terms for piecewise constant test functions, which are essential for ensuring local conservation. To overcome this obstacle, we employ a special projection operator (\Cref{def:2Dprojection}) introduced in \cite{liu2020optimal}. 
	\Cref{prop:welldefiness} establishes that the projection $P_h^*$ is well-defined, and the corresponding local $L^\infty$ stability bounds are derived in \cite[Lemma 3.4]{yang2012analysis} and \cite[Lemma 2.1]{liu2020optimal}. While the operator itself comes from prior work, we identify an additional structure-preserving property---the preservation
 of outflow-edge averages---which is essential for closing the two-dimensional superconvergence analysis; see \Cref{prop:structurepreserve_proj}. This discovery is a cornerstone of our 2D analysis, allowing for the critical cancellation of flux errors for piecewise constant test functions. 
	
	\begin{definition}[2D projection]\label{def:2Dprojection}
		For $\omega \in H^1(\Omega)$, the projection $\Pe  (\omega) \in \mathbb{V}_h^k$ is defined to satisfy the following equations 
		\begin{equation}\label{eq:2Dprojection}
			\begin{cases}
				\ipproj{\proj (\omega)}{v}_{i,j}=\ipproj{\omega}{v}_{i,j},\ \ & \forall v \in \mathbb{P}^k(I_{i,j}),\ \ \forall i,j, \\
				\ip{\Pe  (\omega) }{1}_{i,j} = \ip{\omega}{1}_{i,j},\ \ & \forall i,j,
			\end{cases}
		\end{equation}
		where
		\begin{equation}\label{eq:0702-1}
			\begin{aligned}
				\ipproj{\omega}{v}_{i,j} = & - { \ip{\omega}{ \bm{\beta}\cdot \nabla  v}_{i,j}} 
				 + \int_{x_{i-1/2}}^{x_{i+1/2}} \omega(x, y_{j+1/2}^-) \big(v(x, y_{j+1/2}^-) - v(x, y_{j-1/2}^+)\big) \, \mathrm{d}x \\
				& + \int_{y_{j-1/2}}^{y_{j+1/2}} \omega(x_{i+1/2}^-, y) \big(v(x_{i+1/2}^-, y) - v(x_{i-1/2}^+, y)\big) \, \mathrm{d}y.
			\end{aligned}
		\end{equation}
	\end{definition}

	From \eqref{eq:0702-1}, we can prove the following proposition.
	\begin{proposition}\label{prop:welldefiness}
		When $d=2$, if $F \in \mathbb{V}_h^{k}$ satisfies
		$
		\ip{F}{1}_{i,j}=0$ for all $i,j$, 
		then there exists a unique $\omega \in \mathbb{V}_h^k$ such that
		\begin{subequations}
			\begin{align}
				\ip{\omega}{1}_{i,j} & =0, \\
				\ipproj{\omega}{v}_{i,j} & = \ip{F}{v}_{i,j}\quad \forall v \in \mathbb{P}^k\left( I_{i,j} \right), 
			\end{align}
		\end{subequations}
		for every $i,j$. Moreover, $\omega$ can be bounded by $F$ in the sense that
		\begin{equation}
			\nm{\omega}_{L^{\infty}\left( I_{i,j} \right)} \leq C \nm{F}_{i,j}\quad \forall i,j.
		\end{equation}
		Here, constant $C>0$ is independent of the choice of $F$.
	\end{proposition}
	
\begin{proof}
	The well-posedness of this problem was originally established in \cite[Lemma 2.1]{liu2020optimal}. Therefore, we only present a sketch of the proof here. Using an affine transformation, it suffices to consider the reference cell $[-1,1]^2$. We select the standard Legendre polynomial basis $\{v_q\}$. The projection is equivalent to a linear system involving the matrix $A_{[-1,1]^2}$, defined by
	\[
	A_{\cdot,0} = [1,0,\dots,0]^T\text{ and } [A_{[-1,1]^2}]_{p,q} = \ipproj{v_p}{v_q}_{[-1,1]^2}, \quad p,q > 0.
	\]
	Following \cite[Lemma 2.1]{liu2020optimal}, we consider a function $g = \sum b_p v_p$ in the kernel. The condition $\ipproj{g}{v}=0$ implies that $g$ is periodic and satisfies $\bm{\beta}\cdot \nabla g = 0$ , hence $g$ admits the representation $g(x,y) = p(x - y)$ for some $p \in \mathbb{P}^k([-1,1]).$ 
	The periodic boundary condition further implies the periodicity of $p$, satisfying $p(x-1) = p(x+1)$ and $p(1-y) = p(-1-y)$. Since $p$ is a polynomial, it therefore must be constant. Moreover, due to the zero-mean constraint on $g$ (inherited from $\omega$), we must have $p \equiv 0$, and hence $g \equiv 0$. This implies the coefficient vector $b$ is zero, proving that $A_{[-1,1]^2}$ is invertible. Consequently, since the operator is defined on a finite-dimensional space, the norm of the inverse is bounded, i.e., $\|A_{[-1,1]^2}^{-1}\|_{\ell^\infty} \leq C$.
	
	Finally, we establish stability on an arbitrary cell $C_{i,j}$. Applying the affine mapping to $[-1,1]^2$ yields the linear system:
$\hat{A}_{[-1,1]^2}^T b(\omega) = f$, where $f_p = \frac{1}{h}\ip{F}{v_p}_{i,j}.$ 
	Here, $\hat{A}_{[-1,1]^2}$ is a scaled version of $A_{[-1,1]^2}$ with scaling factors depending only on the mesh ratios $h_x/h$ and $h_y/h$. Therefore, $\hat{A}_{[-1,1]^2}$ maintains the invertibility and uniform boundedness of its inverse. We deduce the following uniform bound:
	\[
	\|\omega\|_{L^\infty(I_{i,j})} \leq \|b(\omega)\|_{\ell^\infty} \leq \|\hat{A}_{[-1,1]^2}^{-1}\|_{\ell^\infty}\|f\|_{\ell^\infty} \leq C\|F\|_{i,j}.
	\]
	This completes the proof.
\end{proof}

	Under Definition \ref{def:2Dprojection}, $\Pe $ also has certain approximation  and  superconvergence properties, as shown below.
	
	\begin{remark}
		The above 2D projection $\Pe $ is chosen because we observe that it preserves both cell averages and outflow edge averages. This property will be crucial in our analysis, as it guarantees certain constant-test-function terms vanish exactly (cf. identity \eqref{eq:0721-4}).
	\end{remark}

	\begin{lemma}\label{lmm:2D-proj-property}
		For $\Pe $ defined by \eqref{eq:2Dprojection}, consider sufficiently smooth $\omega \in H^{k+1}(\Omega)$. Then there exists constant $C>0$ independent of the choice of $\omega$ such that
		\begin{align*}
				\nm{\omega - \Pe  (\omega) }_{i,j} & \leq Ch^{m} \nm{\omega}_{H^{m}(I_{i,j})}\quad m=1,\dots,k+1,\ \ \forall i,j,
				\\
				\left| H\left( \omega - \Pe  (\omega), v \right) \right| & \leq Ch^{k+1} \nm{\omega}_{H^{k+1}\left(\Omega\right)}  \nm{v} \quad \forall v \in \mathbb{V}_h^k.
		\end{align*}
	\end{lemma}
	
	Unlike the 1D case, in this case, it is possible that $H_{i,j}\left( \Pe  (\omega),v \right) \neq H_{i,j}\left( \omega,v \right)$. However, by selecting $v=\frac{2}{h_x}\left(x-\frac{x_{i-\frac{1}{2}}+x_{i+\frac{1}{2}}}{2}  \right)$ and $v=\frac{2}{h_y}\left(y-\frac{y_{j-\frac{1}{2}}+y_{j+\frac{1}{2}}}{2}  \right)$ in \eqref{eq:2Dprojection}, we 
	{\bf have a crucial observation} stated in the following proposition.
	
	\begin{proposition}[Structure-preserving property of projection operator] \label{prop:structurepreserve_proj}
		The projection $\Pe $ preserves cell averages. Moreover, it also 
		preserves the outflow-edge averages of $\omega$ in the sense that
		\begin{align}
			\int_{y_{j-1/2}}^{y_{j+1/2}} (\Pe \omega)(x_{i+1/2}^{-}, y) \, \mathrm{d}y &= \int_{y_{j-1/2}}^{y_{j+1/2}} \omega(x_{i+1/2}^{-}, y) \, \mathrm{d}y,
			\label{equ:up1}\\
			\int_{x_{i-1/2}}^{x_{i+1/2}} (\Pe \omega)(x, y_{j+1/2}^{-}) \, \mathrm{d}x &= \int_{x_{i-1/2}}^{x_{i+1/2}} \omega(x, y_{j+1/2}^{-}) \, \mathrm{d}x.\label{equ:up2}	
		\end{align}
		As a consequence,  
		\begin{equation}\label{eq:0721-4}
			H_{i,j}(\Pe \omega,1) = H_{i,j}(\omega,1)\quad \forall i,j.
		\end{equation}
	\end{proposition} 
	
%
%
%
	
	\subsubsection{Correction functions} 
	We now proceed with Step 2 of our proof sketch: the construction of the correction functions 
	$\left\{ Z^{n,m} \right\}$. These functions are the central analytical tool of our proof, designed to cancel the low-order error terms that obstruct a direct superconvergence estimate. 
	This construction yields the corrected source terms $\left\{ \widetilde{\mathcal{S}}^{n,l+1} \right\}$ while satisfying $\nm{Z^{n,m}} \leq Ch^{k+2}$. Each $Z^{n,m}$ consists of two parts
	\begin{equation}
		Z^{n,m} = Z^{n,m}_1 + Z^{n,m}_2,
	\end{equation}
	where $Z^{n,m}_1,Z^{n,m}_2$ are defined as follows. 
	
	\begin{definition}[Correction functions for 1D problem]\label{def:1Dcorrectfun}
		For $d=1$, we let
		\begin{equation*}
			V_{i}^{n,m}:= \sum_{p=0}^{k} \delta_{I_i}^p\left(u_h^{n,m}\right) \, \left( \Pe  U^{n,m} - P^{(p-1)}\Pe  U^{n,m} \right),
		\end{equation*}
		then $Z^{n,m}_1, Z^{n,m}_2$ are DG polynomials in $\mathbb{V}_h^k$ which satisfy
		\begin{subequations}
			\begin{equation*}
				Z^{n,m}_1\left( x_{i+\frac{1}{2}}^- \right) = Z^{n,m}_2\left( x_{i+\frac{1}{2}}^- \right) =0,
			\end{equation*}
			\begin{equation*}
				\ipproj{Z^{n,m}_1}{v}_i = \ip{ V_i^{n,m}}{v}_i\quad \forall v \in \mathbb{P}^{k}(I_i),
			\end{equation*}
			\begin{equation*}
				\ipproj{Z^{n,m}_2}{v}_i = \ip{\frac{\mathrm{d} U^{n,m}}{\mathrm{d}x}-\Pe \frac{\mathrm{d} U^{n,m}}{\mathrm{d}x}}{v}_i + H_i\left( \eta^{n,m},v \right)\quad \forall v \in \mathbb{P}^{k}(I_i),
			\end{equation*}
		\end{subequations}
		for all $i$.
	\end{definition}
	
	\begin{remark}[Roles of the correction function components]
		The correction function $Z^{n,m}$ is composed of two distinct parts, each targeting a different source of low-order error. The component $Z_{2}^{n,m}$ is analogous to the correction functions used in the superconvergence analysis of standard linear RKDG schemes \cite{xu2020superconvergence}, designed to cancel errors arising from the numerical flux and the projection of the derivative. The component $Z_{1}^{n,m}$ is tailored specifically to counteract the new, nonlinear error terms generated by the OE procedure. It is motivated by the superconvergence analysis for the 1D semi-discrete OFDG method in {\cite{lu2021oscillation}}. 
		This decomposition allows us to isolate and systematically eliminate the novel analytical difficulties posed by the OEDG formulation.
	\end{remark}

	Following  {\cite{lu2021oscillation}} and \cite{xu2020superconvergence}, one has the explicit expressions for $Z_1^{n,m}$ and $Z_2^{n,m}$, as summarized in Lemma \ref{lmm:1Dcorrectfun-expression}.

	\begin{lemma}\label{lmm:1Dcorrectfun-expression}
		For 1D case, define the elementwise antiderivative operator $D^{-1}:\mathbb{V}_h^{k} \to \mathbb{V}_h^{k+1}$ as
		\begin{equation*}
			\left( D^{-1}\omega \right) \big|_{I_i} (x) = \int_{x_{i-\frac{1}{2}}}^{x} \omega(s)\, \mathrm{d}s\quad \forall i.
		\end{equation*}
		Under Definition \ref{def:1Dcorrectfun}, if we set $V^{n,m}$ to be the DG polynomial such that
		\begin{equation}\label{eq:0710-2}
			V^{n,m} \big|_{I_i} = V_i^{n,m}\quad \forall i,
		\end{equation}
		then $Z_1^{n,m}$ and $Z_2^{n,m}$ can be expressed explicitly as
		\begin{equation*}
			Z_1^{n,m} = -\Pe D^{-1}V^{n,m},\ \ Z_2^{n,m}=\Pe D^{-1}P^{(k)}\left( \frac{\mathrm{d} U^{n,m}}{\mathrm{d}x}-\Pe \frac{\mathrm{d} U^{n,m}}{\mathrm{d}x}  \right).
		\end{equation*}
	\end{lemma}
	
	Based on Definition \ref{def:1Dcorrectfun}, we define the correction functions for 2D problem as follows. 
	
	\begin{definition}[Correction functions for 2D problem]\label{def:2Dcorrectfun}
		For $d=2$, we set
		\begin{equation*}
			V_{i,j}^{n,m}:=\sum_{p=0}^{k} \delta_{I_{i,j}}^p\left(u_h^{n,m}\right) \, \left( \Pe  U^{n,m} - P^{(p-1)}\Pe  U^{n,m} \right),
		\end{equation*}
		then we define $Z^{n,m}_1, Z^{n,m}_2 \in \mathbb{V}_h^k$ by, for every $i,j$,
		\begin{subequations}
			\begin{equation*}
				\int_{I_{i,j}} Z_1^{n,m} \, \mathrm{d}x\mathrm{d}y = \int_{I_{i,j}} Z_2^{n,m} \, \mathrm{d}x\mathrm{d}y = 0,
			\end{equation*}
			\begin{equation*}
				\ipproj{Z^{n,m}_1}{v}_{i,j} = \ip{V_{i,j}^{n,m}}{v}_{i,j}\quad \forall v \in \mathbb{P}^{k}(I_{i,j}),
			\end{equation*}
			\begin{equation*}
				\ipproj{Z^{n,m}_2}{v}_{i,j} = \ip{{	 \bm{\beta}\cdot \nabla U^{n,m}}-\Pe \left({	 \bm{\beta}\cdot \nabla U^{n,m}} \right)}{v}_{i,j} + H_{i,j}\left( \eta^{n,m},v \right)\quad \forall v \in \mathbb{P}^{k}(I_{i,j}).
			\end{equation*}
		\end{subequations}
		.
	\end{definition}
	Due to \Cref{prop:welldefiness}, $Z^{n,m}_1$ and $Z^{n,m}_2$ in \Cref{def:2Dcorrectfun} are well-defined. 
	
	We finish this subsection by proving the following estimate for both the 1D and 2D correction functions.
	
	\begin{proposition}[Estimate for correction functions]\label{prop:est_correct}
		Under the CFL condition ${\frac{\tau}{h^{\kappa}}\leq C_{\textrm{CFL}}}$, $\left\{ Z^{n,m} \right\}$ with both Definition \ref{def:1Dcorrectfun} and \ref{def:2Dcorrectfun} satisfy that
		\begin{equation}
			\nm{Z^{n,m}} \leq \nm{Z_1^{n,m}} + \nm{Z_2^{n,m}} \leq C\left( h^{k+2} + \tau^r \right)\quad \forall n,m,
		\end{equation}
		Here constant $C>0$ is independent of $n$, $m$, $\tau$ and $h$.
		
	\end{proposition}
\begin{proof}
The bounds follow from the $L^\infty$ stability of $\Pe$, inverse inequalities on each cell, Conditions~1--2, and Lemma~\ref{lmm:OEerr}.
For any cell $K$, we have 
$$\|Z_1^{n,m}\|_K\leq Ch^{\frac{d}{2}}\|Z_1^{n,m}\|_{L^\infty(K)}\leq Ch\|V^{n,m}\|_K\leq Ch^{k+2}.$$
Also since there exists $F_{\eta^{n,m}}\in \mathbb{V}_h^k$ such that
$$(F_{\eta^{n,m}},v)_K = H_K(\eta^{n,m},v)\quad \forall v\in \mathbb{V}_h^k$$
the $L^\infty$ boundedness of $\tilde{P}$, we have
$$\|Z_2^{n,m}\|_K \leq Ch^{\frac{d}{2}}\|Z_2^{n,m}\|_{L^\infty}\leq Ch\|F_{\eta^{n,m}}\|_K\leq Ch^{k+2}\|U^{n,m}\|_{H^{k+2}(\Omega)}$$.
\end{proof}

	\subsubsection{Proof for Theorem \ref{thm:superconvergence}} \label{sec:sup_analysis}
	With the projection operator and correction functions in hand, we advance to Steps 3 and 4. We introduce the corrected error $\tilde{\zeta}_{\sigma}^{n,m} = \zeta_{\sigma}^{n,m} + Z^{n,m}$ and derive the evolution equation it satisfies. The core of the proof then lies in carefully estimating the resulting source terms to establish the desired $(k+2)$th-order bound.  
	Denote 
	\begin{equation}\label{eq:notation}
			\omega_c^{n,l+1} := \frac{1}{\tau} \left(\omega^{n,l+1} - \sum_{0 \leq m \leq l} c_{lm} \omega^{n,m}\right), \qquad
			\omega_d^{n,l+1} := \sum_{0 \leq m \leq l} d_{lm} \omega^{n,m},
	\end{equation}
	for $\left\{ \omega^{n,m} \right\}$. For $l=0,\dots,s-1$, we can find $\mathcal{S}^{n,l+1}_1, \mathcal{S}^{n,l+1}_2 \in \mathbb{V}_h^k$ such that
		\begin{align*}
			\ip{\mathcal{S}^{n,l+1}_1}{v} &= \ip{\eta_c^{n,l+1}}{v} - H\left( \eta_d^{n,l+1},v  \right)-\ip{\rho^{n,l+1}}{v},\\
			\ip{\mathcal{S}^{n,l+1}_2}{v} &= \ip{\frac{\mathcal{F}_{\hat\tau_{l+1}}u_h^{n,l+1} - u_h^{n,l+1}}{\tau}}{v},
		\end{align*}
	for all $v \in \mathbb{V}_h^k$. Then using \eqref{eq:notation0}, we observe from \eqref{eq:sd-RKOEDG-advec} and \eqref{eq:RKrefsol} that $\left\{ \zeta_{\sigma}^{n,m} \right\}$ satisfies
	\begin{subequations}\label{eq:0630-1}
		\begin{align}
			\zeta_{\sigma}^{n,0} & = \zeta_{\sigma}^{n} = u_{\sigma}^n-\Pe U^n, \\
			\ip{\zeta_{\sigma}^{n,l+1}}{v} & = \sum_{m=0}^l \big(  c_{lm}\ip{\zeta_{\sigma}^{n,m}}{v} + \tau d_{lm} H\left( \zeta_{\sigma}^{n,m}, v \right) \big) + \tau\ip{\mathcal{S}^{n,l+1}}{v}, \label{eq:0630-2} \\
			\zeta_{\sigma}^{n+1} & = \zeta_{\sigma}^{n,s} = u_\sigma^{n+1}-\Pe U^{n+1},        \end{align}
	\end{subequations}
	where the source term $\mathcal{S}^{n,l+1} = \mathcal{S}^{n,l+1}_1 + \mathcal{S}^{n,l+1}_2 \in \mathbb{V}_h^k$.
	
	By Definition \ref{def:1Dcorrectfun}, \ref{def:2Dcorrectfun}, the correction functions satisfies $Z^{n+1,0}=Z^{n,s}$ for each $n$. Now if we set $Z^n := Z^{n,0}$ and introduce $\left\{ Z^{n,m} \right\}$ into \eqref{eq:0630-1}, we obtain the following RKDG scheme for $\left\{ \widetilde{\zeta}_{\sigma}^{n,m}:=\zeta_{\sigma}^{n,m}+Z^{n,m} \right\}$.
		\begin{align*}
			\widetilde{\zeta}_{\sigma}^{n,0} & =\widetilde{\zeta}_{\sigma}^{n} = \zeta_{\sigma}^{n}+Z^n, \\
			\ip{\widetilde{\zeta}_{\sigma}^{n,l+1}}{v} & = \sum_{m=0}^l \big(  c_{lm}\ip{\widetilde{\zeta}_{\sigma}^{n,m}}{v} + \tau d_{lm} H\left( \widetilde{\zeta}_{\sigma}^{n,m}, v \right) \big) + \tau\ip{\widetilde{\mathcal{S}}^{n,l+1}}{v}, \\
			\widetilde{\zeta}_{\sigma}^{n+1} & = \widetilde{\zeta}_{\sigma}^{n,s} = \zeta_{\sigma}^{n+1}+Z^{n+1}.
		\end{align*}
	Here, the corrected source term $\widetilde{\mathcal{S}}^{n,l+1} \in \mathbb{V}_h^k$ can be decomposed into four parts
	\begin{equation*}
		\widetilde{\mathcal{S}}^{n,l+1}=\widetilde{\mathcal{S}}^{n,l+1}_1+\widetilde{\mathcal{S}}^{n,l+1}_2+\widetilde{\mathcal{S}}^{n,l+1}_3+\widetilde{\mathcal{S}}^{n,l+1}_4.
	\end{equation*}
	For all $v \in \mathbb{V}_h^k$, $\widetilde{\mathcal{S}}^{n,l+1}_1,\dots,\widetilde{\mathcal{S}}^{n,l+1}_4 \in \mathbb{V}_h^k$ satisfy
	\begin{subequations}\label{eq:0630-3}
		\begin{align}
			\ip{\widetilde{\mathcal{S}}^{n,l+1}_1}{v} & = \ip{\eta_c^{n,l+1}}{v} - H\left( \eta_d^{n,l+1},v  \right)-\ip{\rho^{n,l+1}}{v} + \ipproj{\left( Z_2 \right)_d^{n,l+1}}{v},\label{eq:0701-1} \\
			\ip{\widetilde{\mathcal{S}}^{n,l+1}_2}{v} & = - H\left( Z_d^{n,l+1},v \right) - \ipproj{Z_d^{n,l+1}}{v},\label{eq:0701-2} \\
			\ip{\widetilde{\mathcal{S}}^{n,l+1}_3}{v} & = \ip{Z_c^{n,l+1}}{v},\label{eq:0701-3} \\
			\ip{\widetilde{\mathcal{S}}^{n,l+1}_4}{v} & = \ip{\frac{\mathcal{F}_{\hat\tau_{l+1}}u_h^{n,l+1} - u_h^{n,l+1}}{\tau}}{v} + \ipproj{\left( Z_1 \right)_d^{n,l+1}}{v}.\label{eq:0701-4}
		\end{align}
	\end{subequations}
	In \eqref{eq:0630-3}, $\ipproj{\omega}{v}=\sum_{i} \ipproj{\omega}{v}_i$ for 1D problem and for 2D case, $\ipproj{\omega}{v} = \sum_{i,j} \ipproj{\omega}{v}_{i,j}$.
	
	To complete our proof, we only need to verify that
	\begin{equation*}
		\nm{\widetilde{\mathcal{S}}^{n,l+1}_1} + \cdots + \nm{\widetilde{\mathcal{S}}^{n,l+1}_4} \leq C\left(\nm{\widetilde{\zeta}_{\sigma}^{n}} + h^{k+2}+\tau^r \right).
	\end{equation*}
{	\begin{remark}[What breaks without RK alignment]\label{rem:why-align}
		The cancellation in the estimate of $\widetilde S_4^{n,\ell+1}$ hinges on the identity
		$\hat\tau_{\ell+1}/\tau=\sum_{m=0}^{\ell}d_{\ell m}$.
		If one uses a stage-independent OE pseudo-step $\hat\tau$ (as in the original OEDG formulation),
		then the factor $\hat\tau/\tau$ cannot match the cumulative RK weights $\sum_{m\le\ell}d_{\ell m}$,
		and a residual term of the form
		$\big(\hat\tau/\tau-\sum_{m\le\ell}d_{\ell m}\big)\sum_p\delta_I^p(\cdot)(\cdot)$ persists.
		This residual is precisely the low-order interface contribution that obstructs the analysis for the $(k+2)$ superconvergence.
	\end{remark}	}
	
	\textbf{Estimation for $\nm{\widetilde{\mathcal{S}}^{n,l+1}_1}$.}
	According to \eqref{eq:refsol} and \eqref{eq:lte},
	\begin{equation*}
		\rho^{n,l+1} = U_c^{n,l+1} + {{	 \bm{\beta}\cdot \nabla  U_d^{n,l+1}}}.
	\end{equation*}
	Therefore,
		\begin{align*}
			\ip{\widetilde{\mathcal{S}}^{n,l+1}_1}{v} & =\ip{\eta_c^{n,l+1}}{v}+\ip{{	 \bm{\beta}\cdot \nabla  U_d^{n,l+1}}-\Pe \left({	 \bm{\beta}\cdot \nabla  U_d^{n,l+1}} \right)}{v}-\ip{\rho^{n,l+1}}{v} \\
			& = \ip{\rho^{n,l+1}-\Pe \rho^{n,l+1}}{v}-\ip{\rho^{n,l+1}}{v} = -\ip{\Pe \rho^{n,l+1}}{v}.
		\end{align*}
	Using the $L^{\infty}$-stability of $\Pe $, we can obtain that
	\begin{equation*}
		\nm{\widetilde{\mathcal{S}}^{n,l+1}_1} \leq C\nm{\Pe \rho^{n,l+1}} \leq C\nm{\Pe \rho^{n,l+1}}_{L^\infty(\Omega)} \leq C\nm{\rho^{n,l+1}}_{L^\infty(\Omega)} \leq C\tau^r.
	\end{equation*}
	
	\textbf{Estimation for $\nm{\widetilde{\mathcal{S}}^{n,l+1}_2}$.} 
	When $d=1$, Definition \ref{def:1Dcorrectfun} yields that
	\begin{equation*}
		H_i\left( Z^{n,m}, v \right) = \ip{Z^{n,m}}{\frac{\mathrm{d}v}{\mathrm{d}x}}_i = - \ipproj{Z^{n,m}}{v}_i,
	\end{equation*}
	for every $i$ and $v \in \mathbb{V}_h^k$. Thus, in this case,
	\begin{equation*}
		\nm{\widetilde{\mathcal{S}}^{n,l+1}_2}^2 = - H\left( Z_d^{n,l+1},\widetilde{\mathcal{S}}^{n,l+1}_2 \right) - \ipproj{Z_d^{n,l+1}}{\widetilde{\mathcal{S}}^{n,l+1}_2} =0.
	\end{equation*}
	On the other hand, if $d=2$, we {\bf make a key observation} that, after the correction, the residual has a discrete-shifting structure: for $v \in \mathbb{P}^k\left( I_{i,j} \right)$,
	\begin{equation}\label{equ:diffstru}
		\begin{aligned}
			-H_{i,j}(\omega,v) - \ipproj{\omega}{v}_{i,j} & =  \int_{y_{j-\frac{1}{2}}}^{y_{j+\frac{1}{2}}} \Delta_x\omega\left( x_{i-\frac{1}{2}}^-,y \right)\, v\left( x_{i-\frac{1}{2}}^+,y \right) \, dy \\
			& + \int_{x_{i-\frac{1}{2}}}^{x_{i+\frac{1}{2}}} \Delta^y\omega\left( x,y_{j-\frac{1}{2}}^- \right)\, v\left( x,y_{j-\frac{1}{2}}^+ \right) \, dx.
		\end{aligned}
	\end{equation}
	By the standard inverse inequality $\|v\|_{L^\infty(I_{i,j})}\le C h^{-1}\|v\|_{i,j}$ for $v\in\mathbb{P}^k(I_{i,j})$, we then derive 
	\begin{align*}
		\nm{\widetilde{\mathcal{S}}^{n,l+1}_2}^2 & \leq Ch\sum_{i,j}\left( \nm{\Delta_x Z_d^{n,l+1}}_{L^\infty\left( I_{i,j} \right)} + \nm{\Delta^y Z_d^{n,l+1}}_{L^\infty\left( I_{i,j} \right)} \right)\nm{\widetilde{\mathcal{S}}^{n,l+1}_2}_{L^\infty\left( I_{i,j} \right)} \\
		& \leq C\sum_{i,j}\left( \nm{\Delta_x Z_d^{n,l+1}}_{L^\infty\left( I_{i,j} \right)} + \nm{\Delta^y Z_d^{n,l+1}}_{L^\infty\left( I_{i,j} \right)} \right)\nm{\widetilde{\mathcal{S}}^{n,l+1}_2}_{i,j} \\
		& \leq C \left( \sum_{i,j}\nm{\Delta_x Z_d^{n,l+1}}_{L^\infty\left( I_{i,j} \right)}^2 + \sum_{i,j}\nm{\Delta^y Z_d^{n,l+1}}_{L^\infty\left( I_{i,j} \right)}^2 \right)^\frac{1}{2}\nm{\widetilde{\mathcal{S}}^{n,l+1}_2}.
	\end{align*}
	Therefore, we obtain that
	\begin{equation}\label{eq:0704-1}
		\begin{aligned}
			\nm{\widetilde{\mathcal{S}}^{n,l+1}_2} & \leq C\left( \sum_{i,j}\nm{\Delta_x Z_d^{n,l+1}}_{L^\infty\left( I_{i,j} \right)}^2 + \sum_{i,j}\nm{\Delta^y Z_d^{n,l+1}}_{L^\infty\left( I_{i,j} \right)}^2 \right)^\frac{1}{2} \\
			& \leq C\sum_{m=0}^{l}\left( \sum_{i,j}\nm{\Delta_x Z^{n,m}}_{L^\infty\left( I_{i,j} \right)}^2 + \sum_{i,j}\nm{\Delta^y Z^{n,m}}_{L^\infty\left( I_{i,j} \right)}^2 \right)^\frac{1}{2}.
		\end{aligned}
	\end{equation}
	Observe that for $v \in \mathbb{P}^k\left(  I_{i,j} \right)$, $\mathbb{P}_h^*\big( Z^{n,m}( x+h_x, y ),v \big)_{i,j} = \mathbb{P}_h^*\big( Z^{n,m}, v( x-h_x, y ) \big)_{i+1,j}$.
	We can derive from Definition \ref{def:2Dcorrectfun} that
		\begin{align*}
			\ipproj{\Delta_x Z^{n,m}}{v}_{i,j} & =  - \ip{V_{i+1,j}^{n,m}(x+h_x,y)-V_{i,j}^{n,m}}{v}_{i,j}  \\
			+ & \ip{{{	 \bm{\beta}\cdot \nabla   \Delta_x U^{n,m}}}-\Pe \left({	 \bm{\beta}\cdot \nabla   \Delta_x U^{n,m}} \right)}{v}_{i,j} + H_{i,j}\left( \Delta_x\eta^{n,m},v \right). 
		\end{align*}
	At this point, the approximation property of $\Pe $  and difference quotient estimate for $\Delta_xU^{n,m}$ in \cite[Theorem~3, p.~292]{evans2022partial} yields that
		\begin{align*}
			\left\|{	 \bm{\beta}\cdot \nabla   \Delta_x U^{n,m}}  -\Pe \left({	 \bm{\beta}\cdot \nabla   \Delta_x U^{n,m}} \right) \right\|_{i,j} & \leq Ch^{k+1} \nm{\Delta_x U^{n,m}}_{H^{k+1}\left( I_{i,j} \right)}  \\ 
			&\leq Ch^{k+2}\nm{ U^{n,m}}_{H^{k+2}\left(\hat{I}_{i,j} \right)},
		\end{align*}
	where $\hat I_{i,j}$ denotes the union of $I_{i,j}$ and its neighboring cells. And take $F_{\Delta_x\eta} \in \mathbb{V}_h^k$ such that $\ip{F_{\Delta_x\eta}}{v}_{i,j}=H_{i,j}\left( \Delta_x\eta^{n,m},v \right)$ for all $i,j$, the superconvergence property of $\Pe $ gives
	\begin{equation*}
			\nm{F_{\Delta_x\eta}} \leq Ch^{k+1}\nm{\Delta_x U^{n,m}}_{H^{k+1}\left( \Omega \right)} \leq Ch^{k+2}\nm{ U^{n,m}}_{H^{k+2}(\Omega)}.
	\end{equation*}
	Using Proposition \ref{prop:welldefiness}, we further show
		\begin{align*}
			\sum_{i,j} \nm{\Delta_x Z^{n,m}}_{L^\infty\left( I_{i,j} \right)}^2 & \leq Ch^{2k+4} + C \sum_{i,j}\nm{V_{i+1,j}^{n,m}(x+h_x,y)-V_{i,j}^{n,m}}_{i,j}^2 \\
			 & \leq Ch^{2k+4} + C \sum_{i,j}\nm{V_{i+1,j}^{n,m}-V_{i,j}^{n,m}}_{i+1,j}^2 + C \sum_{i,j}\nm{\Delta_x V_{i,j}^{n,m}}_{i,j}^2 .
		\end{align*}
	Now, by the definition of $V_{i,j}^{m,n}$,
	\begin{equation}\label{eq:0708-2}
		\begin{aligned}
			& \sum_{i,j}\nm{V_{i+1,j}^{n,m}-V_{i,j}^{n,m}}_{i+1,j}^2 \\
			& \leq C\sum_{p=0}^{k}\sum_{i,j}\left( \delta_{I_{i+1,j}}^p\left( u_h^{n,m} \right) - \delta_{I_{i,j}}^p\left( u_h^{n,m} \right) \right)^2 \nm{\Pe  U^{n,m} - P^{(p-1)}\Pe  U^{n,m}}_{i+1,j}^2 \\
			& \leq C\sum_{p=0}^{k}\sum_{i,j}\left( \delta_{I_{i+1,j}}^p\left( u_h^{n,m} \right) - \delta_{I_{i,j}}^p\left( u_h^{n,m} \right) \right)^2\left( \nm{\eta^{n,m}}_{i+1,j} + \nm{U^{n,m}-P^{(0)}U^{n,m}}_{i+1,j} \right)^2 \\
			& \leq C\sum_{p=0}^{k}\sum_{i,j}\left( \delta_{I_{i+1,j}}^p\left( u_h^{n,m} \right) - \delta_{I_{i,j}}^p\left( u_h^{n,m} \right) \right)^2 \left( h \nm{U^{n,m}}_{H^{1}(I_{i+1,j})} \right)^2 \leq C\left( h^{2k+4} + \tau^{2r} \right).
		\end{aligned}
	\end{equation}
	The second to last inequality is derived by the observation that
	\begin{equation}\label{eq:0708-1}
		\nm{\eta^{n,m}}_{i,j} + \nm{U^{n,m}-P^{(0)}U^{n,m}}_{i,j} \leq Ch\nm{U^{n,m}}_{H^1(I_{i,j})} ,
	\end{equation}
	for every $i,j$. In the meantime, based on \eqref{eq:0708-1},
		\begin{align*}
			\sum_{i,j}\nm{\Delta_x V_{i,j}^{n,m}}_{i,j}^2 & \leq C\sum_{p=0}^{k}\sum_{i,j}\left(  \delta_{I_{i,j}}^p\left( u_h^{n,m} \right) \right)^2 \nm{\Delta_x \left(  \Pe  U^{n,m} \right) - P^{(p-1)}\left( \Delta_x \Pe  U^{n,m} \right)}_{i+1,j}^2 \\
			\leq C \sum_{p=0}^{k} & \sum_{i,j} \left(  \delta_{I_{i,j}}^p\left( u_h^{n,m} \right) \right)^2 \left( \nm{\Delta_x \eta^{n,m}}_{i,j} + \nm{\Delta_x U^{n,m} - P^{(0)}\left( \Delta_x U^{n,m} \right) }_{i,j} \right)^2 \\
			\leq C \sum_{p=0}^{k} & \sum_{i,j} \left(  \delta_{I_{i,j}}^p\left( u_h^{n,m} \right) \right)^2 \left( { h^2 \nm{U^{n,m}}_{H^{2}(\hat{I}_{i,j})} }\right)^2 \leq C\left( h^{2k+4} + \tau^{2r} \right).
		\end{align*}
	Therefore, we can conclude that
	\begin{equation*}
		\sum_{i,j} \nm{\Delta_x Z^{n,m}}_{L^\infty\left( I_{i,j} \right)}^2 \leq C\left( h^{2k+4} + \tau^{2r} \right),  
	\end{equation*}
	and similarly,
	\begin{equation*}
		\sum_{i,j} \nm{\Delta^y Z^{n,m}}_{L^\infty\left( I_{i,j} \right)}^2 \leq C\left( h^{2k+4} + \tau^{2r} \right).  
	\end{equation*}
	As a result, by \eqref{eq:0704-1},
	\begin{equation*}
		\nm{\widetilde{\mathcal{S}}^{n,l+1}_2} \leq C\left( h^{2k+4} + \tau^{2r} \right)^\frac{1}{2} \leq C\left( h^{k+2} + \tau^{r} \right).
	\end{equation*}

	\textbf{Estimation for $\nm{\widetilde{\mathcal{S}}^{n,l+1}_3}$.} 
	Our estimate of \(\nm{\widetilde{\mathcal{S}}^{n,l+1}_3}\) relies on the local Lipschitz-type condition \eqref{eq:lipschitz-typecontinuity2}, which yields
	\begin{equation}\label{eq:0709-1}
		\resizebox{0.99\hsize}{!}{$
			\begin{aligned}
				\sum_{K \in \mathcal{T}_h}
				\bigl(\delta_K^p(u_h^{n,l+1})-\delta_K^p(u_h^{n,m})\bigr)^2
				&\le Ch^{2k-d}\,\bigl(\|u_h^{n,l+1}-u_h^{n,m}\|_{L^{\infty}(\Omega)}^2+\nm{U^{n,l+1}-U^{n,m}}_{H^{k+1}\left( \Omega \right)}^2\bigr)
				\\
				&+ Ch^{-2-d} \nm{\zeta^{n,l+1}-\zeta^{n,m}}^2.
			\end{aligned}$}
	\end{equation}

	From \eqref{eq:0701-3}, we know that
	\begin{equation*}
		\nm{\widetilde{\mathcal{S}}^{n,l+1}_3} = \nm{Z_c^{n,l+1}} \leq \nm{\left( Z_1 \right)_c^{n,l+1}} + \nm{\left( Z_2 \right)_c^{n,l+1}}.
	\end{equation*}
	In 1D case ($d=1$), Lemma \ref{lmm:1Dcorrectfun-expression} indicates that
	\begin{subequations}
		\begin{align}
			\left( Z_1 \right)_c^{n,l+1} & = -\Pe D^{-1}V_c^{n,l+1}, \\
			\left( Z_2 \right)_c^{n,l+1} & = \Pe D^{-1}P^{(k)}\left( \frac{\mathrm{d} U_c^{n,l+1}}{\mathrm{d}x}-\Pe \frac{\mathrm{d} U_c^{n,l+1}}{\mathrm{d}x}  \right).
		\end{align}
	\end{subequations}
	By the definition of $D^{-1}$ and $P^{(k)}$, and the $L^\infty$-stability of $\Pe $, we can apply inverse inequalities and obtain that for every $i$,
		\begin{align*}
			\nm{\left( Z_2 \right)_c^{n,l+1}}_i & \leq Ch^{\frac{1}{2}}\nm{\left( Z_2 \right)_c^{n,l+1}}_{L^{\infty}\left( I_i \right)} \\
			& \leq Ch^\frac{3}{2} \nm{P^{(k)}\left( \frac{\mathrm{d} U_c^{n,l+1}}{\mathrm{d}x}-\Pe \frac{\mathrm{d} U_c^{n,l+1}}{\mathrm{d}x}  \right)}_{L^{\infty}\left( I_i \right)} \\
			& \leq Ch\nm{P^{(k)}\left( \frac{\mathrm{d} U_c^{n,l+1}}{\mathrm{d}x}-\Pe \frac{\mathrm{d} U_c^{n,l+1}}{\mathrm{d}x}  \right)}_i \\
			& \leq Ch\nm{\frac{\mathrm{d} U_c^{n,l+1}}{\mathrm{d}x}-\Pe \frac{\mathrm{d} U_c^{n,l+1}}{\mathrm{d}x}}_i \leq Ch^{k+2}\nm{U_c^{n,l+1}}_{H^{k+2}\left( I_i \right)}.
		\end{align*}
	Thus we have,
	\begin{equation*}
		\nm{\left( Z_2 \right)_c^{n,l+1}} \leq Ch^{k+2}\nm{U_c^{n,l+1}}_{H^{k+2}\left( \Omega \right)} \leq Ch^{k+2}\sum_{l=0}^{s-1}\frac{\nm{U^{n,l+1}-U^{n,0}}_{H^{k+2}\left( \Omega \right)}}{\tau}.
	\end{equation*}
	For $l\leq s-2$, combining \eqref{eq:0709-2} with the induction hypothesis yields
	\begin{equation}\label{eq:0721-2}
		\nm{U^{n,l+1}-U^{n,0}}_{H^{k+2}\left( \Omega \right)} \leq C\tau\sum_{m=0}^l\nm{U^{n,m}}_{H^{k+2} \left( \Omega \right)}\leq C\tau,
	\end{equation}
	and if $l=s-1$,
	\begin{equation}\label{eq:0721-3}
		\nm{U^{n,l+1}-U^{n,0}}_{H^{k+2}\left( \Omega \right)} = \nm{U^{n+1}-U^{n}}_{H^{k+2}\left( \Omega \right)} \leq C\tau.
	\end{equation}
	Therefore, $\nm{\left( Z_2 \right)_c^{n,l+1}} \leq Ch^{k+2}$.
	
	We can also use the arguments above and show that $\nm{\left( Z_1 \right)_c^{n,l+1}}_i \leq Ch\nm{V_c^{n,l+1}}_i$. By the procedure of \eqref{eq:0708-2},
		\begin{align*}
			\nm{V_c^{n,l+1}}_i & \leq C\sum_{p=0}^k\sum_{m=0}^{l}\frac{\left| \delta_{I_i}^p\left(u_h^{n,l+1}\right) - \delta_{I_i}^p\left(u_h^{n,m}\right) \right|}{\tau} \cdot \left( \nm{\eta^{n,m}}_i + \nm{U^{n,m}-P^{(0)}U^{n,m}}_i \right) \\
			& +  C\sum_{p=0}^k\left| \delta_{I_i}^p\left(u_h^{n,l+1}\right) \right| \cdot\left( \nm{\eta_c^{n,l+1}}_i + \nm{U_c^{n,l+1}-P^{(0)}U_c^{n,l+1}}_i \right) \\
			& \leq Ch^\frac{3}{2}\sum_{p=0}^k\left( \sum_{m=0}^{l}\frac{\left| \delta_{I_i}^p\left(u_h^{n,l+1}\right) - \delta_{I_i}^p\left(u_h^{n,m}\right) \right|}{\tau} + \left| \delta_{I_i}^p\left(u_h^{n,l+1}\right) \right|  \right).
		\end{align*}
	Then, combining inequality \eqref{eq:0709-1} with Lemmas \ref{lmm:OEerr} and \ref{lmm:deltaOEerr-1}, we obtain the following technical estimate 
	\begin{equation}\label{eq:WW1}
		\begin{aligned}
			\nm{\left( Z_1 \right)_c^{n,l+1}} & \leq C\sum_{m=0}^{l}\frac{h}{\tau}\nm{\zeta^{n,l+1}-\zeta^{n,m}} + C\sum_{m=0}^{l}\frac{h^{k+2}}{\tau}\nm{U^{n,l+1}-U^{n,m}}_{H^{k+1}\left( \Omega \right)} \\
			& + C\sum_{m=0}^{l}\frac{h\left( h^{k+1} + \tau^{r} \right)}{\tau}\nm{u_h^{n,l+1} - u_h^{n,m}}_{L^\infty(\Omega)}  + C\left( h^{k+2} + \tau^r \right).
		\end{aligned}
	\end{equation}
	{\bf A crucial observation is that the factor $\frac{h}{\tau}$ appearing in \eqref{eq:WW1} will be  exactly canceled by $\frac{\tau}{h}$ arising in \eqref{eq:0709-3}, forming a key point in our analysis.} 
	For $\nm{\zeta^{n,l+1}-\zeta^{n,m}}$, Lemma \ref{lmm:OEerr} indicates 
		\begin{align*}
			\nm{\zeta^{n,l+1}-\zeta^{n,m}} & \leq \nm{\zeta_\sigma^{n,l+1}-\zeta_\sigma^{n,m}} + \nm{u_\sigma^{n,m} - u_h^{n,m}} + \nm{\mathcal{F}_{\hat\tau_{l+1}}u_h^{n,l+1} - u_h^{n,l+1}}  \\
			& \leq \nm{\zeta_\sigma^{n,l+1}-\zeta_\sigma^{n,0}} + \nm{\zeta_\sigma^{n,m}-\zeta_\sigma^{n,0}} + C\tau \left( h^{k+1} + \tau^r \right).
		\end{align*}
	Notice from \eqref{eq:0630-2} that
		\begin{align*}
			\ip{\zeta_{\sigma}^{n,l+1}-\zeta_{\sigma}^{n,0}}{v} & = \sum_{m=0}^l \Big(  c_{lm}\ip{\zeta_{\sigma}^{n,m}-\zeta_{\sigma}^{n,0}}{v} + \tau d_{lm} H\left( \zeta_{\sigma}^{n,m}-\zeta_{\sigma}^{n,0}, v \right) \Big) \\
			& + \tau \left( \sum_{m=0}^ld_{lm} \right) H\left( \zeta_{\sigma}^{n,0}, v \right) + \tau\ip{\mathcal{S}^{n,l+1}}{v}.
		\end{align*}
	Since $\left| H\left( \omega , v \right) \right| \leq Ch^{-1}\nm{\omega}\nm{v}$ for all $\omega,v \in \mathbb{V}_h^k$ and $\nm{\mathcal{S}^{n,l+1}} \leq C\left( h^{k+1} + \tau^r \right) $, by inductions, we can verify under the constraint ${\frac{\tau}{h^\kappa} \leq C_{\text{CFL}}}$ that
	\begin{equation*}
		\nm{\zeta_\sigma^{n,m}-\zeta_\sigma^{n,0}} \leq C\frac{\tau}{h}\nm{\zeta_\sigma^{n,0}} + C \tau\left( h^{k+1} + \tau^r \right)\quad \forall m =0 ,\dots,s,\ \ \forall n.
	\end{equation*}
	As $\zeta_\sigma^{n,0} = \zeta_\sigma^{n}$, we conclude that 
	\begin{equation}\label{eq:0709-3}
		\nm{\zeta^{n,l+1}-\zeta^{n,m}} \leq C\frac{\tau}{h}\nm{\zeta_\sigma^{n}} + C \tau\left( h^{k+1} + \tau^r \right).
	\end{equation}
	For $\nm{U^{n,l+1}-U^{n,m}}_{H^{k+1}\left( \Omega \right)}$, similar to \eqref{eq:0721-2} and \eqref{eq:0721-3},
	\begin{equation}\label{eq:0709-4}
			\nm{U^{n,l+1}-U^{n,m}}_{H^{k+1}\left( \Omega \right)} \leq \nm{U^{n,l+1}-U^{n,0}}_{H^{k+1}\left( \Omega \right)} + \nm{U^{n,m}-U^{n,0}}_{H^{k+1}\left( \Omega \right)} \leq C\tau.
	\end{equation}
	For $\nm{u_h^{n,l+1} - u_h^{n,m}}_{L^\infty(\Omega)}$, observe that there exists $i_0$ such that
	\begin{equation}\label{eq:0710-1}
		\nm{u_h^{n,l+1} - u_h^{n,m}}_{L^\infty(\Omega)} = \nm{u_h^{n,l+1} - u_h^{n,m}}_{L^\infty\left(I_{i_0}\right)}.
	\end{equation}
	\eqref{eq:0710-1} implies that
	\begin{equation}\label{eq:0709-5}
		\begin{aligned}
			\nm{u_h^{n,l+1} - u_h^{n,m}}_{L^\infty(\Omega)} & \leq Ch^{-\frac{1}{2}}\nm{\zeta^{n,l+1}-\zeta^{n,m}}_{i_0} + C \nm{\Pe  \left( U^{n,l+1} - U^{n,m} \right) }_{L^\infty\left(I_{i_0}\right)} \\
			& \leq Ch^{-\frac{1}{2}} \nm{\zeta^{n,l+1}-\zeta^{n,m}} + C\nm{U^{n,l+1} - U^{n,m}}_{L^\infty(\Omega)} \\
			& \leq Ch^{-\frac{1}{2}} \nm{\zeta^{n,l+1}-\zeta^{n,m}} + C\tau
		\end{aligned}
	\end{equation}
	Combining \eqref{eq:0709-3}, \eqref{eq:0709-4}, and \eqref{eq:0709-5}, we obtain a technical estimate: 
		\begin{align*}
			\nm{\left( Z_1 \right)_c^{n,l+1}} & \leq C\sum_{m=0}^{l}\frac{ h+h^{k+\frac{3}{2}}+ h^{r+\frac{1}{2}} }{\tau}\nm{\zeta^{n,l+1}-\zeta^{n,m}} + C\left( h^{k+2} + \tau^r \right) \\
			& \leq C\sum_{m=0}^{l}\frac{ h }{\tau}\nm{\zeta^{n,l+1}-\zeta^{n,m}} + C\left( h^{k+2} + \tau^r \right) \\
			& \leq C\left( \nm{\zeta_{\sigma}^n} + h^{k+2} + \tau^r \right) \leq C\left( \nm{\widetilde{\zeta}_{\sigma}^n} + h^{k+2} + \tau^r \right).
		\end{align*}
	Therefore, for $d=1$,
	\begin{equation}\label{eq:0709-6}
		\nm{\widetilde{\mathcal{S}}^{n,l+1}_3} \leq C\left( \nm{\widetilde{\zeta}_{\sigma}^n} + h^{k+2} + \tau^r \right).
	\end{equation}
	
	For 2D problem ($d=2$), it follows from \Cref{prop:welldefiness} that
	\begin{equation*}
		\nm{\left( Z_1 \right)_c^{n,l+1}}_{i,j} \leq Ch\nm{\left( Z_1 \right)_c^{n,l+1}}_{L^{\infty}\left(  I_{i,j} \right)} \leq Ch\nm{V_c^{n,l+1}}_{i,j},
	\end{equation*}
	and let $F_{\eta_c} \in \mathbb{V}_h^k$ satisfy $\ip{F_{\eta_c}}{v}_{i,j} = H_{i,j}\left( \eta_c^{n,l+1},v \right)$ for all $i,j$,
		\begin{align*}
			\nm{\left( Z_2 \right)_c^{n,l+1}} \leq Ch\nm{\left( Z_2 \right)_c^{n,l+1}}_{L^{\infty}\left(  I_{i,j} \right)} & \leq Ch \left( \nm{\nabla U_c^{n,l+1}-\Pe \left( \nabla U_c^{n,l+1} \right)} + \nm{F_{\eta_c}} \right) \\
			& \leq Ch^{k+2}\nm{U_c^{n,l+1}}_{H^{k+2}\left( \Omega \right)} \leq Ch^{k+2}.
		\end{align*}
	Then, similarly to the case $d=1$, we can prove \eqref{eq:0709-6} for $d=2$.

	\textbf{Estimation for $\nm{\widetilde{\mathcal{S}}^{n,l+1}_4}$}
	Here we only consider the estimation for $\nm{\widetilde{\mathcal{S}}^{n,l+1}_4}$ when $d=1$. Because $\widetilde{\mathcal{S}}^{n,\ell+1}_4$ is the cellwise correction associated with the OE step, the argument is local and carries over directly to $d=2$ on each cell $C_{i,j}$ and its faces. The 2D case is therefore omitted.
	
	According to Definition \ref{def:1Dcorrectfun} and \eqref{eq:0710-2},
	\begin{equation*}
		\ip{\widetilde{\mathcal{S}}^{n,l+1}_4}{v} = \ip{\frac{\mathcal{F}_{\hat\tau_{l+1}}u_h^{n,l+1} - u_h^{n,l+1}}{\tau}}{v} + \ip{V_d^{n,l+1}}{v},
	\end{equation*}
	for all $v \in \mathbb{V}_h^k$. If we substitute $v=v_i$ where $v_i \in \mathbb{V}_h^k$ is defined by
	\begin{equation*}
		v_i(x) = \begin{cases}
			\widetilde{\mathcal{S}}^{n,l+1}_4(x),\quad & x \in I_i, \\
			0,\quad & otherwise,
		\end{cases}
	\end{equation*}
	we can obtain that for every $i$,
	\begin{equation}\label{eq:0710-3}
		\nm{\widetilde{\mathcal{S}}^{n,l+1}_4}_i^2 = \ip{\frac{\mathcal{F}_{\hat\tau_{l+1}}u_h^{n,l+1} - u_h^{n,l+1}}{\tau}}{v_i}_i + \ip{\left( V_i \right)_d^{n,l+1}}{v_i}_i.
	\end{equation}
	Let $u_\sigma\left( \hat{t} \right)$ be the solution of the following initial value problem
	\begin{equation*}
		\begin{cases}
			\frac{\mathrm{d}}{\mathrm{d}\hat{t}} \ip{u_{\sigma}}{v}_i = -\sum_{p=0}^{k}\delta_{I_i}^p\left( u_h^{n,l+1} \right) \ip{u_\sigma - P^{(p-1)}u_{\sigma}}{v}_i, \\
			u_{\sigma}(0) = u_h^{n,l+1},
		\end{cases}
	\end{equation*}
	for all $i$ and $v \in \mathbb{P}^k\left( I_i \right)$, then $\mathcal{F}_{\hat\tau_{l+1}}u_h^{n,l+1} = u_{\sigma}\left( \hat\tau_{l+1} \right)$. Using the mean value theorem for integration, we can find $0<\hat{s}_i<\hat\tau_{l+1}$ such that
		\begin{align*}
			\ip{\frac{\mathcal{F}_{\hat\tau_{l+1}}u_h^{n,l+1} - u_h^{n,l+1}}{\tau}}{v_i}_i & = \frac{\int_{0}^{\hat\tau_{l+1}}\frac{\mathrm{d}}{\mathrm{d}\hat{t}} \ip{u_{\sigma}}{v_i}_i\, \mathrm{d}\hat{t}}{\tau} = \frac{\hat\tau_{l+1}}{\tau}\left( \frac{\mathrm{d}}{\mathrm{d}\hat{t}} \ip{u_{\sigma}}{v_i}_i \bigg|_{\hat{t}=\hat{s}_i} \right) \\
			&= - \sum_{m=0}^{l} d_{lm}\sum_{p=0}^{k}\delta_{I_i}^p\left( u_h^{n,l+1} \right) \ip{u_\sigma\left( \hat{s}_i \right) - P^{(p-1)}u_{\sigma}\left( \hat{s}_i \right)}{v_i}_i.
		\end{align*}
	Then with \eqref{eq:0710-3},
		\begin{align*}
			\nm{\widetilde{\mathcal{S}}^{n,l+1}_4}_i^2 & = \sum_{m=0}^{l} d_{lm}\sum_{p=0}^{k}\delta_{I_i}^p\left( u_h^{n,l+1} \right) \ip{ u_\sigma^{n,l+1}-u_{\sigma}\left( \hat{s}_i \right) - P^{(p-1)}\left( u_\sigma^{n,l+1}-u_{\sigma}\left( \hat{s}_i \right) \right)}{v_i}_i \\
			&- \sum_{m=0}^{l} d_{lm}\sum_{p=0}^{k}\delta_{I_i}^p\left( u_h^{n,l+1} \right)\ip{\zeta_{\sigma}^{n,l+1}-P^{(p-1)}\zeta_{\sigma}^{n,l+1}}{v_i}_i  \\
			&- \sum_{m=0}^{l} d_{lm}\sum_{p=0}^{k}\delta_{I_i}^p\left( u_h^{n,l+1} \right) \ip{\Pe \left(U^{n,l+1}-U^{n,m}\right) - P^{(p-1)}\Pe \left(U^{n,l+1}-U^{n,m}\right)}{v_i}_i \\
			&- \sum_{m=0}^{l} d_{lm}\sum_{p=0}^{k} \left( \delta_{I_i}^p\left( u_h^{n,l+1} \right) -\delta_{I_i}^p\left( u_h^{n,m} \right) \right) \ip{\Pe U^{n,m}-P^{(p-1)}\Pe U^{n,m}}{v_i}_i.
		\end{align*}
	Since $\nm{v_i}_i=\nm{\widetilde{\mathcal{S}}^{n,l+1}_4}_i$,
		\begin{align*}
			& \nm{\widetilde{\mathcal{S}}^{n,l+1}_4}_i \leq C\sum_{p=0}^{k} \left| \delta_{I_i}^p\left( u_h^{n,l+1} \right) \right| \cdot \left( \nm{u_\sigma^{n,l+1}-u_{\sigma}\left( \hat{s}_i \right)}_i + \nm{\zeta_{\sigma}^{n,l+1}}_i \right) \\
			& + C\sum_{m=0}^l\sum_{p=0}^{k} \left| \delta_{I_i}^p\left( u_h^{n,l+1} \right) \right| \cdot \left( \nm{\eta^{n,l+1}-\eta^{n,m}}_i + \nm{U^{n,l+1}-U^{n,m}-P^{(0)}\left(U^{n,l+1}-U^{n,m}\right)}_i \right) \\
			& +  C\sum_{m=0}^l\sum_{p=0}^{k} \left| \delta_{I_i}^p\left( u_h^{n,l+1} \right) - \delta_{I_i}^p\left( u_h^{n,m} \right) \right| \cdot \left( \nm{\eta^{n,m}}_i + \nm{U^{n,m}-P^{(0)}U^{n,m}}_i \right).
		\end{align*}
	Notice that
	\begin{equation*}
		\nm{u_\sigma^{n,l+1}-u_{\sigma}\left( \hat{s}_i \right)}_i \leq \nm{\mathcal{F}_{\hat\tau_{l+1}}u_h^{n,l+1} - u_h^{n,l+1}} \leq C\tau\left( h^{k+1} + \tau^r \right),
	\end{equation*}
	\begin{equation*}
		\nm{\eta^{n,m}}_i + \nm{U^{n,m}-P^{(0)}U^{n,m}}_i \leq Ch\nm{U^{n,m}}_{H^1\left(I_i\right)} \leq Ch^{\frac{3}{2}}\nm{U^{n,m}}_{W^{1,\infty}\left(\Omega\right)},
	\end{equation*}
	and by \eqref{eq:refsol},
	\begin{equation*}
		\begin{aligned}
			& \nm{\eta^{n,l+1}-\eta^{n,m}}_i + \nm{U^{n,l+1}-U^{n,m}-P^{(0)}\left(U^{n,l+1}-U^{n,m}\right)}_i \\
			& {\leq Ch\left( \nm{U^{n,l+1}-U^{n,0}}_{H^{1}\left(I_i\right)} + \nm{U^{n,m}-U^{n,0}}_{H^{1}\left(I_i\right)} \right) \leq C\tau h^{\frac{3}{2}}.}
		\end{aligned}
	\end{equation*}
	Meanwhile, during the proof of \Cref{prop:opterror}, we need to show that when ${\frac{\tau}{h^{\kappa}}\leq C_{\text{CFL}}}$,
	\begin{equation*}
		\nm{\zeta_{\sigma}^{n,l+1}} \leq \nm{\zeta_\sigma^n} + \nm{\zeta_\sigma^{n,l+1}-\zeta_\sigma^{n,0}}
		\leq C\nm{\zeta_\sigma^n} + C\tau\left( h^{k+1} + \tau^r \right),
	\end{equation*}
	 and
	\begin{equation*}
		\begin{aligned}
			h^{\frac{3}{2}}\left( \sum_i\sum_{m=0}^l\sum_{p=0}^{k} \left( \delta_{I_i}^p\left( u_h^{n,l+1} \right) - \delta_{I_i}^p\left( u_h^{n,m} \right) \right)^2 \right)^{\frac{1}{2}} & \leq C\frac{\tau}{h}\left( \nm{\zeta_\sigma^n} + h^{k+2} + \tau^r \right) \\
			& \leq C\nm{\zeta_\sigma^n} + C\tau\left( h^{k+1} + \tau^r \right).
		\end{aligned}
	\end{equation*}
	Therefore, by Lemma \ref{lmm:OEerr},
		\begin{align*}
			\nm{\widetilde{\mathcal{S}}^{n,l+1}_4} & \leq C\left( \sum_i\sum_{p=0}^{k} \left( \delta_{I_i}^p\left( u_h^{n,l+1} \right) \right)^2 \right)^\frac{1}{2} \left( \nm{\zeta_\sigma^n} + \tau\left( h^{k+1} + \tau^r \right) \right) \\
			& + C\tau h^{\frac{3}{2}}\left( \sum_i\sum_{p=0}^{k} \left( \delta_{I_i}^p\left( u_h^{n,l+1} \right) \right)^2 \right)^\frac{1}{2} \\
			& + Ch^{\frac{3}{2}}\left( \sum_i\sum_{m=0}^l\sum_{p=0}^{k} \left( \delta_{I_i}^p\left( u_h^{n,l+1} \right) - \delta_{I_i}^p\left( u_h^{n,m} \right) \right)^2 \right)^{\frac{1}{2}} \\
			& \leq C\bigg( 1 + h^{k-\frac{1}{2}} + \left(\frac{\tau}{h} \right)^r h^{r-\frac{3}{2}} \bigg)\left( \nm{\zeta_\sigma^n} + \tau\left( h^{k+1} + \tau^r \right) \right) \leq C\left( \nm{\widetilde{\zeta}_\sigma^n} +  h^{k+2} + \tau^r \right).
		\end{align*}

	\section{Extension to the linear variable-coefficient advection equation \label{sec:VC}}
Throughout this section we set $d\le 2$ and assume that the wind direction does not reverse; i.e., the wave speeds $\bm{\beta}=(\beta^1,\beta^2)$ keep  fixed signs over $(0,T)$ in each cell: for convenience and without loss of generality, we assume {$\beta^1>0$ and $\beta^2>0$}. Furthermore, we also assume $\|\nabla\!\cdot \bm{\beta}\|_{L^\infty(\Omega)}\le C$. The model problem is
\begin{equation}\label{eq:VC:PDE}
	\partial_t u + \nabla\!\cdot\!\big(\bm{\beta}(\bm{x})\,u\big)=0
	\qquad (\bm{x},t)\in \Omega\times(0,T],
\end{equation}
where $\bm{\beta}:\Omega\to\mathbb{R}^d$ is a smooth velocity field.	In this section we extend the superconvergence theory of Section~\ref{sec:superconvergence} to the smooth
variable-coefficient linear advection equation in one and two dimensions.
We use the same finite element space and upwind numerical fluxes as in the main text.
We assume that the velocity field $\beta(x)$ has a fixed (uniform) wind direction,
$\min_{1\le i\le d} \beta_i(x)>0$ for all $x\in\Omega$, and $\|\nabla\!\cdot\bm{\beta}\|_{L^\infty(\Omega)}$ is bounded.
\subsection{$L^2$ stability of the RKDG method}
We denote a generic constant $C>0$ which may depend on $\|\bm{\beta}\|_{W^{1,\infty}(\Omega)}$ and $\|\nabla\!\cdot\bm{\beta}\|_{L^\infty(\Omega)}$, and on the mesh shape-regularity, but is independent of $h$.
The DG spatial operator for linear variable-coefficient advection equations enjoys the following properties. 
\begin{lemma}[Weak boundedness]\label{lem:VC:weakbdd}
	For all $w,v\in \mathbb{V}_h^k$,
	\[
	|H(w,v)| \;\le\; C\,h^{-1}\,\|w\|\,\|v\|.
	\]
\end{lemma}

\begin{lemma}[Quasi skew-symmetry and $L^2$ stability]\label{lem:VC:skewsym}
	Let $\mathcal{T}_h$ be a {shape-regular} partition of $\Omega$ and, for each cell $C\in\mathcal{T}_h$, let $\bm n_{\partial C}$ denote the outward unit normal on $\partial C$ and $\llbracket\cdot\rrbracket$ the scalar jump. Then, for all $w,v\in \mathbb{V}_h^k$,
	\begin{align*}
		H(w,v)+H(v,w)
		&= -\int_{\Omega} (\nabla\!\cdot \bm{\beta})\, w\,v\,\mathrm{d}\bm{x}
		\;-\; \sum_{C\in \mathcal{T}_h}\int_{\partial C}\!\big|\bm{\beta}\!\cdot\!\bm{n}_{\partial C}\big|\,\llbracket w\rrbracket\,\llbracket v\rrbracket\,\mathrm{d}s .
	\end{align*}
	Furthermore, by taking $w$ as $v$, we obtain the $L^2$ stability for the semidiscrete scheme. There exists $C>0$ such that the semidiscrete solution $u_h(t)$ satisfies
	\[
	\frac{d}{dt}\,\|u_h(t)\|^2 \;\le\; C\,\|u_h(t)\|^2 .
	\]
	
\end{lemma}

\subsubsection{Matrix transferring techniques for the proof of Lemma \ref{lem:VC:stab}} \label{ass:transfer} 

Write
\begin{equation}\label{equ:energy}
	\|u_h^{n+m}\|^2-\|u_h^{n}\|^2
	=\sum_{0\le i,j\le{ms-1}} a^{(\ell)}_{ij}\,(\mathbb D_i u_h^{n},\mathbb D_j u_h^{n})
	\;+\; \tau\sum_{0\le i,j\le {ms-1}} b^{(\ell)}_{ij}\,H(\mathbb D_i u_h^{n},\mathbb D_j u_h^{n}),
\end{equation}
with symmetric $A^{(\ell)}:=\big(a_{ij}^{(\ell)}\big)$ and $B^{(\ell)}:=\big(b_{ij}^{(\ell)}\big)$.
Although the coefficient is variable, the equation we consider is still linear. Thus, the same matrix transferring techniques (3.4) in \cite{xu2020superconvergence} can be applied to \eqref{equ:energy}.  Symmetry is preserved at each transfer. After $L$ steps, the first $L$ columns of $A^{(L)}$
are zero. By Lemma~\ref{lem:VC:skewsym} and Lemma~\ref{lem:VC:weakbdd}, the matrix-transfer/energy argument in Lemma~\ref{assump:stability} carries over with a growth factor depending on $\|\nabla\!\cdot\bm{\beta}\|_{L^\infty}$; see, e.g., \cite[Theorem~3.1]{xu2019l2} and \cite[Proposition~3.2]{ai20222}.

\begin{lemma}[Energy stability]\label{lem:VC:stab}
	Consider the RKDG scheme with stage sources $\{g^{n,m}\}$. We set $u_h^{0,0} = u_h^{0}$ and $	u_h^{n+1,0} = u_h^{n,s}$
	\begin{subequations}
		\begin{align*}
			\ip{u_h^{n,\ell+1}}{v}
			&= \sum_{m=0}^{\ell}\Big(c_{\ell m}\,\ip{u_h^{n,m}}{v}
			+ \tau\, d_{\ell m}\, H\!\left(u_h^{n,m},v\right)\Big)
			+ \tau\,\ip{g^{n,\ell+1}}{v},\ \ell=0,\dots,s-1.
		\end{align*}
	\end{subequations}
	with $v\in\mathbb{V}_h^{k}$ and $\{g^{n,\ell+1}\}\subset\mathbb{V}_h^{k}$. Then there exists $\kappa\ge 1$ such that, under the CFL restriction $\tau\le C_{\mathrm{CFL}}\,h^{\kappa}$,
	$
	\|u_h^{n+1}\|^{2}\;\le\; \bigl(1+C_{\mathrm{s}}\tau\bigr)\,\|u_h^{n}\|^{2}
	\;+\; C\,\tau\sum_{\ell=0}^{s-1}\|g^{n,\ell+1}\|^{2},
	$
	where $C_{\mathrm{s}}$ depends on $\|\nabla\!\cdot \bm{\beta}\|_{L^{\infty}(\Omega)}$ but not on $h$ or $\tau$. If the time integrator is an $r$-stage, $r$th-order RK method, one can choose 
	$\kappa$ as in \cite{xu2020superconvergence,sun2019strong}.
\end{lemma}

\begin{proof}
	Throughout, $Q(\lambda)$ denotes a polynomial of $\lambda:=\tau/h$.
	We follow the matrix-transfer framework of \cite{xu2019l2}, with the explicit transfer
	recurrences stated in \Cref{ass:transfer} of temporal difference defined  by
	\begin{equation}\label{equ:tempordiff}
		(\mathbb{D}_{0}u_h^{\,n},v):=(u_h^{\,n},v),\qquad
		(\mathbb{D}_{q}u_h^{\,n},v):=\tau\,H(\mathbb{D}_{q-1}u_h^{\,n},v)\quad(q\ge1).
	\end{equation}
	By definition, the stage solution can be expressed as a linear combination of the temporal differences,
	\begin{equation}\label{eq:key_energy_identity}
		u_h^{\,n+m}=\sum_{l=0}^{ms-1}\alpha_l\,\mathbb{D}_l u_h^{\,n},
	\end{equation}
	which implies
	$
	\|u_h^{\,n+m}\|^2-\|u_h^{\,n}\|^2
	=\sum_{0\le l_1,l_2\le ms} a^{(0)}_{l_1l_2}\,
	\big(\mathbb{D}_{l_1}u_h^{\,n},\,\mathbb{D}_{l_2}u_h^{\,n}\big).
	$ Using (3.4) in \cite{xu2020superconvergence} for exactly ${L}$ transfer steps yields
	\begin{align*}
		\|u_h^{\,n+m}\|^2-\|u_h^{\,n}\|^2
		=\sum_{0\le l_1,l_2\le ms-1} \big[a_{l_1l_2}^{({L})}\,
		\big(\mathbb{D}_{l_1}u_h^{\,n},\mathbb{D}_{l_2}u_h^{\,n}\big)
		\;+\;
		\tau b_{l_1l_2}^{({L})}\,
		H\!\big(\mathbb{D}_{l_1}u_h^{\,n},\mathbb{D}_{l_2}u_h^{\,n}\big)\big].
	\end{align*}
	By construction, after $L$ transfers the first $L$ columns and rows of $A^{(L)}=(a^{(L)}_{l_1l_2})$ vanish,
	and the $\rho\times \rho$ leading principal block of $B^{(L)}=(b^{(L)}_{l_1l_2})$ is symmetric
	positive definite.
	
	\textbf{Estimate of the $A^{(L)}$-part.}
	Using inverse/trace inequalities together with \eqref{equ:tempordiff}, we have
	$$
	\Big|\sum_{0\le l_1,l_2\le ms-1} a_{l_1l_2}^{(L)}\,
	(\mathbb{D}_{l_1}u_h^{\,n},\mathbb{D}_{l_2}u_h^{\,n})\Big|
	\;\le\;
	\big(a_{LL}^{(L)}\,\lambda^{2L}+\lambda^{2L+1}\,Q_1(\lambda)\big)\,
	\|u_h^{\,n}\|^2.
	$$
	
	\textbf{Estimate of the $B^{(L)}$-part.}
	For $w=\mathbb{D}_{l_1}u_h^{\,n}$ and $v=\mathbb{D}_{l_2}u_h^{\,n}$,
	Lemma~\ref{lem:VC:skewsym} gives
	$
	H(w,v)+H(v,w)
	= -\!\int_{\Omega}(\nabla\!\cdot\bm\beta)\,wv\,\mathrm d\bm x
	- \sum_{C\in\mathcal{T}_h}\!\!\int_{\partial C}\!|\bm\beta\!\cdot\bm n|
	\,\llbracket w\rrbracket\,\llbracket v\rrbracket\,\mathrm ds.
	$
	The face term can be estimated following the same procedure as in \cite{xu2019l2}, while the cell integrals are directly bounded by $C\|w\|\|v\|$. Therefore,
	\[
	\Big|\tau\!\!\sum_{0\le l_1,l_2\le ms-1}\! b_{l_1l_2}^{(L)}\,
	H\!\big(\mathbb{D}_{l_1}u_h^{\,n},\mathbb{D}_{l_2}u_h^{\,n}\big)\Big|
	\;\le\; C\,(\tau Q_2(\lambda)+\lambda ^{\min\{2\rho+1,2L\}}\,Q_3(\lambda))\,\|u_h^{\,n}\|^2.
	\]
	Therefore,
	$
	\|u_h^{n+m}\|^2-\|u_h^{n}\|^2 \;\le\; C\,\tau\,Q(\lambda)\,\|u_h^{n}\|^2 \;+\; C\,\lambda^{\gamma}\,\|u_h^{n}\|^{2} \;+\; C\,\tau\sum_{l=0}^{ms-1}\|g^{n,l+1}\|^{2},
	$
	where $\lambda:=\tau/h$, $Q$ is a polynomial, and $\gamma$ is the same exponent as in the constant-coefficient analysis {in \cite{xu2019l2}}. With the scaling $\tau\le C\,h^{{\gamma/(\gamma-1)}}$ {and $\kappa:=\frac{\gamma}{\gamma-1}$}, the first two terms are absorbed into a linear growth factor, yielding
	\[
	\|u_h^{n+m}\|^2-\|u_h^{n}\|^2 \;\le\; C_{\mathrm s}\,\tau\,\|u_h^{n}\|^{2} \;+\; C\,\tau\sum_{l=0}^{ms-1}\|g^{n,l+1}\|^{2}.
	\] 
\end{proof}

This lemma underpins the fully discrete error estimate and the superconvergence analysis.	\subsection{Optimal (discrete shifting) error estimate}

To prove Proposition~\ref{prop:opterror} and Proposition~\ref{prop:deltaerr}, two additional ingredients are required: (i) a reference solution, defined exactly as in the constant-coefficient case of Section~3.1 with the spatial operator replaced accordingly, thus omitted here. And (ii) a suitable projection operator with frozen velocity, which is the key to optimal error bounds and superconvergence.

\subsubsection{A locally frozen-direction projection}\label{subsec:VC:proj}
Because the flux is linear with spatially varying coefficients, directly incorporating the nonuniform speed $\bm{\beta}(\bm{x})$ into the projection is delicate for well-posedness in two dimensions. A standard remedy (see, e.g., \cite{liu2020optimal,U_Cao_2025,jiao2022optimal}) is to freeze the direction locally, replacing $\bm{\beta}(\bm{x})$ by a cellwise representative (such as an outflow edge average or a point value), and to define the projection with respect to this locally frozen velocity. For convenience, in one dimension we may also freeze the wave speed cellwise (this is not strictly necessary; see \cite{cao2018superconvergence}), in which case the projection coincides with the right Gauss--Radau projection. We now detail the two-dimensional construction.

\paragraph{Two-dimensional projection}
For each cell $C_{i,j}$, define the cellwise frozen velocity
\[
\hat{\bm{\beta}}_{i,j} = \bigl(\hat{{\beta}}_{i,j}^1,\hat{{\beta}}_{i,j}^2\bigr),\qquad 
\hat{{\beta}}_{i,j}^1:= \frac{1}{|I_i|}\int_{I_i}\beta_1(x,y_{j+\frac12})\,\mathrm{d}x,\quad
\hat{{\beta}}_{i,j}^2:= \frac{1}{|J_j|}\int_{J_j}\beta_2(x_{i+\frac12},y)\,\mathrm{d}y.
\]
The cellwise frozen velocity satisfies $\|\boldsymbol\beta-\hat{\boldsymbol\beta}\|_{L^\infty}=O(h)$.

\begin{definition}[2D frozen-direction projection \cite{liu2020optimal}]\label{def:2Dproj_linvar}
	For $\omega\in H^1(\Omega)$, the function $\Pe (\omega)\in \mathbb{V}_h^{k}$ is defined cellwise by
	\begin{equation*}
		\begin{cases}
			\ipproj{\Pe (\omega)}{v}_{i,j} = \ipproj{\omega}{v}_{i,j}, & \forall\, v\in \mathbb{P}^{k}(C_{i,j}),\ \forall\, i,j,\\[2pt]
			\ip{\Pe (\omega)}{1}_{i,j} = \ip{\omega}{1}_{i,j}, & \forall\, i,j,
		\end{cases}
	\end{equation*}
	where the bilinear form with frozen velocity is
	\begin{equation}\label{equ:operatorP_linvar}
		\begin{aligned}
			\ipproj{w}{v}_{i,j}
			&:= -\!\ip{w}{\,\hat{\bm\beta}_{i,j}\!\cdot\nabla v\,}_{i,j}
			+ \int_{I_i}\hat{\beta}^{2}_{i,j}\,w(x,y_{j+\frac12}^{-})
			\bigl(v(x,y_{j+\frac12}^{-})-v(x,y_{j-\frac12}^{+})\bigr)\,dx \\
			&\quad + \int_{J_j}\hat{\beta}^{1}_{i,j}\,w(x_{i+\frac12}^{-},y)
			\bigl(v(x_{i+\frac12}^{-},y)-v(x_{i-\frac12}^{+},y)\bigr)\,dy .
		\end{aligned}
	\end{equation}
\end{definition}

\begin{remark}[Projection with frozen velocity]
	Using a cellwise constant weight $\hat{\boldsymbol\beta}_{i,j}$ does not change
	well-posedness, the standard approximation properties, or the $L^\infty$-stability
	of $\widetilde P$. Routine proofs are omitted.
\end{remark}

Moreover, the key identities \eqref{equ:up1}-\eqref{equ:up2} remain valid for the frozen-velocity projection in Lemma~\ref{lem:sturcture_pre_linvar} with identical proof, which implies \eqref{equ:consttest_linvar} in Lemma~\ref{lem:upwind-frozen}.
\begin{lemma}[Edge-average preservation under frozen velocity]\label{lem:sturcture_pre_linvar}
	Let $\Pe$ be the frozen-direction projection in \cref{def:2Dproj_linvar} with cellwise constant
	$\hat{\bm{\beta}}_{i,j}=(\hat\beta^1_{i,j},\hat\beta^2_{i,j})$. Then for every $\omega\in H^1(\Omega)$,
	on each cell $C_{i,j}=I_i\times J_j$,
	\begin{align*}
		\int_{J_j} \Pe(\omega)(x_{i+\frac12}^-,y)\,dy = \int_{J_j} \omega(x_{i+\frac12}^-,y)\,dy, \quad 
		\int_{I_i} \Pe(\omega)(x,y_{j+\frac12}^-)\,dx = \int_{I_i} \omega(x,y_{j+\frac12}^-)\,dx .
	\end{align*}
\end{lemma}
Together with Lemma~\ref{lem:H-consistency} and Proposition~\ref{prop:super-const-test}, these results guarantee that both the optimal convergence rate and the superconvergence rate are preserved despite the coefficient-freezing error in our analytical framework.

\begin{lemma}[{Supercloseness for piecewise constants}]\label{lem:upwind-frozen}
	For every piecewise constant test function $v$ with $v|_{K}\in\mathbb P^{0}(K)$ for all
	$K\in\mathcal T_h$,
	\begin{equation}\label{equ:consttest_linvar}
		H_{\hat{\boldsymbol\beta}}(\eta,v)=0.
	\end{equation}
	Moreover, for all $v\in\mathbb V_h^{k}$,
	$
	\big|H_{\hat{\boldsymbol\beta}}(\eta,v)\big|\le C\,h^{k+1}\,\|v\|.
	$
	The observation of \eqref{equ:diffstru} remains true for $H_{\hat{\bm{\beta}}}$    \begin{equation}\label{equ:diffstru_linvar}
		\begin{aligned}
			-H_{\hat{\bm{\beta}}}(\omega,v) - \sum_{ij}\ipproj{\omega}{v}_{ij} & =  \sum_j\int_{y_{j-\frac{1}{2}}}^{y_{j+\frac{1}{2}}} \Delta_x({\hat{\beta}^1\omega})\left( x_{i-\frac{1}{2}}^-,y \right)\, v\left( x_{i-\frac{1}{2}}^+,y \right) \, dy \\
			& + \sum_i\int_{x_{i-\frac{1}{2}}}^{x_{i+\frac{1}{2}}} \Delta^y({\hat{\beta}^2\omega})\left( x,y_{j-\frac{1}{2}}^- \right)\, v\left( x,y_{j-\frac{1}{2}}^+ \right) \, dx.
		\end{aligned}
	\end{equation}
	Here $H_{\hat{\boldsymbol\beta}}$ is obtained from $H$ by replacing
	$\boldsymbol\beta$ with its cellwise frozen approximation $\hat{\boldsymbol\beta}$.
\end{lemma}

\begin{lemma}[Estimate of approximate spatial operator]\label{lem:H-consistency}
	Since $\|\boldsymbol\beta-\hat{\boldsymbol\beta}\|_{L^\infty}=O(h)$, the approximate spatial operator satisfies, for any $w,v\in\mathbb V_h^{k}$,
	\[
	\big|\,H(w,v)-H_{\hat{\boldsymbol\beta}}(w,v)\,\big|
	\;\le\; C\,\|\boldsymbol\beta-\hat{\boldsymbol\beta}\|_{L^\infty(\Omega)}\,h^{-1}\,\|w\|\,\|v\|
	\;\le\; C\,\|w\|\,\|v\|.
	\]
\end{lemma}

However, the identity \eqref{eq:0721-4} in Proposition~\ref{prop:structurepreserve_proj}—which is crucial for showing that no correction is needed when the test function is piecewise constant—no longer holds in linear variable-coefficient advection equations. We can still obtain a supercloseness property when the test functions are piecewise constants.

\begin{proposition}[Supercloseness for piecewise constants]\label{prop:super-const-test}
	For every piecewise constant
	$v$ with $v|_{K}\in\mathbb P^{0}(K)$ for all $K\in\mathcal T_h$,
	$
	\big|H(\eta,v)\big|\le C\,h^{k+2}\,\|v\|.
	$
\end{proposition}

\begin{proof}
	In the constant–coefficient case, \eqref{equ:consttest_linvar} implies that no correction is needed for piecewise constant $v$. When $\bm{\beta}$ varies in space, the flux contributions on a cell $C_{i,j}=I_i\times J_j$ do not cancel, and for general $w$ and piecewise constant $v$ one has
	\[
	\begin{aligned}
		H(w,v)
		&= \int_{I_i}\!\Big(\beta_2(x,y_{j+\frac12})\,w(x,y_{j+\frac12}^{-})
		- \beta_2(x,y_{j-\frac12})\,w(x,y_{j-\frac12}^{-})\Big)\,dx \\
		&\quad + \int_{J_j}\!\Big(\beta_1(x_{i+\frac12},y)\,w(x_{i+\frac12}^{-},y)
		- \beta_1(x_{i-\frac12},y)\,w(x_{i-\frac12}^{-},y)\Big)\,dy \;\neq\;0 .
	\end{aligned}
	\]
	Introduce the frozen–velocity DG operator $H_{\hat{\boldsymbol\beta}}$ obtained from $H$ by replacing
	$\bm{\beta}$ with its cellwise frozen approximation $\hat{\boldsymbol\beta}$. Then
	$
	H(\eta,v)={H(\eta,v)-H_{\hat{\boldsymbol\beta}}(\eta,v)}
	\;+\;{H_{\hat{\boldsymbol\beta}}(\eta,v)}.
	$
	
	The second term vanishes by Lemma~\ref{lem:upwind-frozen} for piecewise constant $v$.
	For the first term, 
	$|H(\eta,v)-H_{\hat{\bm{\beta}}}(\eta,v)| \leq Ch^{-1}\|\bm{\beta}-\hat{\bm{\beta}}\|_{L^{\infty}}\|v\|(\|\Delta_x \eta\|+\|\Delta^y \eta\|)\leq Ch^{k+2}\|v\|.$
	The last bound follows from $\|\bm{\beta}-\hat{\bm{\beta}}\|_{L^{\infty}}=\mathcal O(h)$ together with \eqref{eq:0720-1}.
	This proves the desired supercloseness for piecewise constant test functions.
	
\end{proof}

\subsubsection{Proof of Proposition~\ref{prop:opterror} and Proposition~\ref{prop:deltaerr}}
In this setting, it suffices to prove Proposition~\ref{prop:opterror}. The proof of Proposition~\ref{prop:deltaerr} follows from Proposition~\ref{prop:opterror} with only minor modifications, the main difference from the linear case being an additional local variation of the wave speed \(\mathbf{\beta}\), which is of order \(O(h)\).

Proceeding as the constant-coefficient case (Theorem~\ref{thm:superconvergence}), for all \(v\in\mathbb V_h^k\), the stage error
\(\zeta_{\sigma}^{n,\ell}:=u_{\sigma}^{n,\ell}-\Pe U^{n,\ell}\) satisfies, 
\begin{equation}\label{eq:VC:error}
	\big(\zeta_{\sigma}^{n,\ell+1},v\big)
	= \sum_{m=0}^{\ell}\!\Big[c_{\ell m}\big(\zeta_{\sigma}^{n,m},v\big)
	+ \tau\, d_{\ell m}\, H\big(\zeta_{\sigma}^{n,m},v\big)\Big]
	+ \tau\,\mathcal{Y}^{n,\ell+1}(v)
	+ \tau\,\mathcal{Z}^{n,\ell+1}(v)
	+ \tau\,\mathcal{X}^{n,\ell+1}(v),
\end{equation}
where
\[
\mathcal{Y}^{n,\ell+1}(v) := \tau^{-1}\big(\mathcal{F}_{\hat{\tau}^{\ell+1}}u_{h}^{n,\ell+1}-u_h^{n,\ell+1},\,v\big),\qquad
\mathcal{Z}^{n,\ell+1}(v):= (\eta_c^{n,\ell+1},v)-H_{\hat{\bm{\beta}}}(\eta_d^{n,\ell+1},v)-(\rho^{n,\ell+1},v),
\]
\[
\mathcal{X}^{n,\ell+1}(v) := H_{\hat{\bm{\beta}}}(\eta_d^{n,\ell+1},v)-H(\eta_d^{n,\ell+1},v).
\]
\paragraph{Estimate of $\mathcal{X}^{n,\ell+1}$}
Lemma~\ref{lem:H-consistency} has shown
\[
\big|\mathcal{X}^{n,\ell+1}(v)\big|
\;\le\; C\,h^{-1}\,\|\bm{\beta}-\hat{\bm{\beta}}\|_{L^\infty(\Omega)}\,\|\eta_d^{n,\ell+1}\|\,\|v\|
\;\le\; C\, h^{k+1}\,\|v\|.
\]

Proceeding as in \cite{peng2025oscillation} and invoking Lemma~\ref{lem:VC:stab} in \eqref{eq:VC:error} yields
\[
\|\zeta^{\,n+1}\|^{2}
\;\le\; \bigl(1+C_s\tau\bigr)\,\|\zeta^{\,n}\|^{2}
\;+\; C\,\tau\Big(\tau^{r}+h^{k+1}\Big)\|\zeta^{\,n}\|,
\]
from which a discrete {Gr\"onwall} argument gives
$
\max_{0\le n\le \lfloor T/\tau\rfloor}\|\zeta^{\,n}\|
\;\le\; C\big(h^{k+1}+\tau^{r}\big),
$
using the initial bound $\|\zeta^{0}\|\le Ch^{k+1}$.

\subsection{Superconvergence analysis}
The extension of the superconvergence analysis to the variable-coefficient case introduces a new technical challenge. The use of a frozen-coefficient projection (Definition \ref{def:2Dproj_linvar}), while necessary for a well-posed definition, creates a consistency error between the true DG operator $H$ and the frozen velocity DG operator $H_{\hat{\beta}}$. This discrepancy introduces error terms that are of order $O(h^{k+1})$ in general, which contaminate the desired $O(h^{k+2})$ superconvergence estimate. The key to overcoming this difficulty lies in a refined analysis of these new error terms. We will show that for the crucial case of piecewise constant test functions, a supercloseness property (Proposition \ref{prop:super-const-test}) holds, demonstrating that these potentially destructive low-order errors are in fact one order higher than they appear. 

\subsubsection{Correction functions}
As in Section~\ref{sec:superconvergence}, write
$
Z^{n,m}=Z^{n,m}_{1}+Z^{n,m}_{2}.
$
In the variable-coefficient setting the operator $\mathbb{P}_h^*$ is modified by a
cellwise frozen weight. The 2D version is given in \eqref{equ:operatorP_linvar}. In 1D, for each $I_i$,
$
\ipproj{w}{v}_{i}
:= -\!\int_{I_i}\hat{\beta}_{i+\frac12}\,w\,v_x\,dx.
$ 
Before defining the correction functions, introduce the zero-mean subspaces
\[
\mathbb{P}_0^{k}(I_i)=\big\{v\in\mathbb{P}^{k}(I_i)\ \big|\ \ip{v}{1}_{i}=0\big\},
\qquad
\mathbb{P}_0^{k}(C_{i,j})=\big\{v\in\mathbb{P}^{k}(C_{i,j})\ \big|\ \ip{v}{1}_{i,j}=0\big\}.
\]

\begin{definition}[Correction functions]\label{def:correct_linvar} In 1D, for each cell $I_i$ define $Z^{n,m}_2\in\mathbb{V}_h^{k}$ by \[ \left\{ \begin{aligned} \ipproj{Z^{n,m}_2}{v}_{i} &= \ip{\dfrac{dU^{n,m}}{dx}-\Pe \!\left(\dfrac{dU^{n,m}}{dx}\right)}{v}_{i} + H_i(\eta^{n,m},v) \quad \forall\, v\in\mathbb{P}_0^{k}(I_i),\\ Z_{2}^{n,m}\!\left(x_{i+\frac12}^{-}\right)&=0. \end{aligned} \right. \] In 2D, for each cell $C_{i,j}$ define $Z^{n,m}_2\in\mathbb{P}_0^{k}$ by \[ \ipproj{Z^{n,m}_2}{v}_{i,j} = \ip{\bm{\beta}\!\cdot\nabla U^{n,m} - \Pe \!\big(\bm{\beta}\!\cdot\nabla U^{n,m}\big)}{v}_{i,j} + H_{i,j}(\eta^{n,m},v) \quad \forall\, v\in\mathbb{P}_0^{k}(C_{i,j}). \] And we follow the way in Definition~\ref{def:1Dcorrectfun} and \ref{def:2Dcorrectfun} to define $Z_1^{n,m}$. \end{definition}

Similar to Proposition~\ref{prop:est_correct}, the functions in Definition~\ref{def:correct_linvar} satisfy the bounds below; the proof follows the exact same steps and is omitted.
\begin{lemma}[Estimate of correction functions]\label{lem:est_correct_linvar}
	\[
	\|Z_d^{n,\ell}\|\le Ch^{k+2},
	\qquad
	\|\Delta_x Z_d^{n,\ell}\|+\|\Delta^y Z_d^{n,\ell}\|\le Ch^{k+3}.
	\]
\end{lemma}

\subsubsection{Proof for Theorem~\ref{thm:superconvergence}}\label{subsec:VC:main}
Proceeding as in the constant-coefficient setting yields the decomposition \eqref{eq:0630-3}. 
The terms $\widetilde{\mathcal{S}}_3^{\,n,\ell+1}$ and $\widetilde{\mathcal{S}}_4^{\,n,\ell+1}$ are determined, respectively, by the RK time integrator and the OE step, so their estimates are identical to the constant-coefficient case. 
It remains to bound $\widetilde{\mathcal{S}}_1^{\,n,\ell+1}$ and $\widetilde{\mathcal{S}}_2^{\,n,\ell+1}$. 
For $\widetilde{\mathcal{S}}_1^{\,n,\ell+1}$ we consider two classes of test functions: (i) piecewise constants and (ii) polynomials with zero cell mean. 
Case (ii) coincides with the constant-coefficient analysis, so we only detail case (i) below.

\paragraph*{Estimate of $\widetilde{\mathcal{S}}_1^{\,n,\ell+1}$ for piecewise constant $v$}
Since $\Pe $ preserves cell averages in both 1D and 2D,
\[
\begin{aligned}
	(\widetilde{\mathcal{S}}_1^{n,\ell+1},v)
	&= -H(\eta_d^{n,\ell+1},v)-(\rho^{n,\ell+1},v) \\
	&= \big(H_{\hat{\bm{\beta}}}-H\big)(\eta_d^{n,\ell+1},v)\;-\;(\rho^{n,\ell+1},v)\;-\;H_{\hat{\bm{\beta}}}(\eta_d^{n,\ell+1},v).
\end{aligned}
\]
By Lemma~\ref{lem:upwind-frozen}, $H_{\hat{\bm{\beta}}}(\eta_d^{n,\ell+1},v)=0$, hence
\[
\big|(\widetilde{\mathcal{S}}_1^{n,\ell+1},v)\big|
\;\le\; C\,(h^{k+2}+\tau^r)\,\|v\|,
\qquad \forall\, v\ \text{with}\ v|_{K}\in \mathbb{P}^0(K),\ K\in \mathcal{T}_h.
\]

\paragraph*{Estimate of $\widetilde{\mathcal{S}}_2^{\,n,\ell+1}$}
Insert and subtract $H_{\hat{\bm{\beta}}}$ to obtain
\[
\big(\widetilde{\mathcal{S}}_2^{\,n,\ell+1},v\big)
= \big(H_{\hat{\bm{\beta}}}-H\big)(Z_d^{n,\ell+1},v)
\;+\;
\Big(\!-H_{\hat{\bm{\beta}}}(Z_d^{n,\ell+1},v)-\ipproj{Z_d^{n,\ell+1}}{v}\Big).
\]
By Lemma~\ref{lem:H-consistency}, \eqref{equ:diffstru_linvar}, and standard inverse inequalities,
\begin{align*}
	\big|H(w,v)-H_{\hat{\bm{\beta}}}(w,v)\big| & \le C\,\|w\|\,\|v\|,
	\\
	\big|H_{\hat{\bm{\beta}}}(w,v)+\ipproj{w}{v}\big| 
	&\le C\,h^{-1}\big(\|\Delta_x w\|+\|\Delta^{y} w\|\big)\|v\|.
\end{align*}
With Lemma~\ref{lem:est_correct_linvar} this yields
\[
\big|(\widetilde{\mathcal{S}}_2^{n,\ell+1},v)\big|
\;\le\; C\,\|Z_d^{n,\ell+1}\|\,\|v\|
+ C\,h^{-1}\big(\|\Delta_x Z_d^{n,\ell+1}\|+\|\Delta^{y} Z_d^{n,\ell+1}\|\big)\|v\|
\;\le\; C\,h^{k+2}\,\|v\|.
\]

At $n=0$ one has $\zeta^{0}=0$ by the choice of initial data, so
$\|\widetilde{\zeta}^{\,0}\| = \|Z^{0}\|\le C h^{k+2}$.

Combining the above bounds with Lemma~\ref{lem:VC:stab} and applying a discrete {Gr\"onwall} argument gives
$
\max_{0\le n\le N}\|\widetilde{\zeta}^{\,n}\|
\;\le\; C\,(h^{k+2}+\tau^{r})$ with $
\widetilde{\zeta}^{\,n}:=\zeta^{\,n}+Z^{n}.
$

\subsection{Detailed proof of the correction-function estimate}
\begin{proposition}[Estimate for correction functions]
	Under the CFL condition ${\frac{\tau}{h^{\kappa}}\leq C_{\textrm{CFL}}}$, $\left\{ Z^{n,m} \right\}$ with both Definition \ref{def:1Dcorrectfun} and \ref{def:2Dcorrectfun} satisfies that
	\begin{equation}
		\nm{Z^{n,m}} \leq \nm{Z_1^{n,m}} + \nm{Z_2^{n,m}} \leq C\left( h^{k+2} + \tau^r \right)\quad \forall n,m,
	\end{equation}
	Here constant $C>0$ is independent of $n$, $m$, $\tau$ and $h$.
	
\end{proposition}
\begin{proof}
	For 1D correction functions, Lemma \ref{lmm:1Dcorrectfun-expression} yields that
	\begin{align*}
		\nm{Z_1^{n,m}}_i &\leq Ch^{\frac{1}{2}}\nm{Z_1^{n,m}}_{L^{\infty}\left( I_i \right)} \leq Ch^{\frac{3}{2}} \nm{V_i^{n,m}}_{L^{\infty}\left( I_i \right)} \leq Ch\nm{V_i^{n,m}}_i,\\
		\nm{Z_2^{n,m}}_i & \leq Ch^{\frac{1}{2}}\nm{Z_2^{n,m}}_{L^{\infty}\left( I_i \right)}  \leq Ch^\frac{3}{2} \nm{P^{(k)}\left( \frac{\mathrm{d} U^{n,m}}{\mathrm{d}x}-\Pe \frac{\mathrm{d} U^{n,m}}{\mathrm{d}x}  \right)}_{L^{\infty}\left( I_i \right)} \\
		& \leq Ch\nm{P^{(k)}\left( \frac{\mathrm{d} U^{n,m}}{\mathrm{d}x}-\Pe \frac{\mathrm{d} U^{n,m}}{\mathrm{d}x}  \right)}_i   \leq Ch^{k+2}\nm{U^{n,m}}_{H^{k+2}\left( I_i \right)}.
	\end{align*}
	Then we obtain that
	\begin{align*}
		\|Z_1^{n,m}\|
		&\le Ch\Bigg(\sum_i\sum_{p=0}^k\big(\delta_{I_i}^p(u_h^{n,m})\big)^2
		\big\|\widetilde P U^{n,m}-P^{(p-1)}\widetilde P U^{n,m}\big\|^2\Bigg)^{1/2} \\
		& \leq Ch^{\frac{5}{2}}\nm{U^{n,m}}_{W^{1,\infty}(\Omega)}\left( \sum_{i}\sum_{p=0}^k\left( \delta_K^p\left( u_h^{n,m} \right) \right)^2 \right)^{\frac{1}{2}} \leq C\left( h^{k+2} + \tau^r \right),
	\end{align*}
	and $\nm{Z_2^{n,m}} \leq Ch^{k+2}\nm{U^{n,m}}_{H^{k+2}\left( \Omega \right)} \leq Ch^{k+2}$.
	
	If $d=2$, applying Proposition \ref{prop:welldefiness} to Definition \ref{def:2Dcorrectfun}, we show that
	$
	\nm{Z_1^{n,m}}_{i,j} \leq Ch\nm{Z_1^{n,m}}_{L^{\infty}\left( I_{i,j} \right)} \leq Ch\nm{V_{i,j}^{n,m}}_{i,j},
	$
	We thus derive that $\nm{Z_1^{n,m}} \leq C\left( h^{k+2} + \tau^r \right)$ following the procedure above. Moreover, since there exists $F \in \mathbb{V}_h^k$ such that
	$\ip{F}{v}_{i,j} = H_{i,j}\left( \eta^{n,m},v \right)$ for all $v \in \mathbb{V}_h^k$ and all $i,j$. 
	The superconvergence property of 2D $\Pe $ indicates that
	\begin{align*}
		\nm{Z_2^{n,m}} \leq Ch\left( \sum_{i,j}\nm{Z_2^{n,m}}_{L^{\infty}\left( I_{i,j} \right)}^2 \right)^\frac{1}{2}
		& \leq Ch^{k+2}\nm{U^{n,m}}_{H^{k+2}\left( \Omega \right)} \leq Ch^{k+2}.
	\end{align*}
\end{proof}

	\section{Numerical tests} \label{sec:numerics}
	In this section, we present numerical experiments in one and two dimensions to verify the $(k+2)$th-order superconvergence of the RK-aligned OE-type DG scheme  and to illustrate its oscillation-control effectiveness. 
	The examples cover the advection equations, the inviscid Burgers' equation, and the compressible Euler equations.
	
	For smooth problems, the superconvergence is examined using the seventh-order linear strong-stability-preserving explicit RK time discretization from \cite{gottlieb1998total} for the $\mathbb{P}^k$- or $\mathbb{Q}^k$-based RK-aligned OE-type DG method with the same damping coefficients as in \cite{O_Peng_2024}. For problems with discontinuities, we employ the classical third-order strong-stability-preserving explicit RK scheme for the $\mathbb{P}^k$-based RK-aligned OE-type DG method. The time step is chosen as $\Delta t = C_{\text{CFL}} h / \beta$ for 1D problems, where $\beta$ denotes the maximum wave speed, and as $\Delta t = C_{\text{CFL}} / (\beta_x / h_x + \beta_y / h_y)$ for 2D problems, where $\beta_x$ and $\beta_y$ are the maximum wave speeds in the $x$- and $y$-directions, respectively. For both the $\mathbb{P}^k$- and $\mathbb{Q}^k$-based RK-aligned OE-type DG methods, the CFL number is set to $1 / (2k + 1)$.

	The upwind numerical flux is used for linear equations, while the local Lax-Friedrichs flux is adopted for the remaining cases.  To further assess the ability to eliminate spurious oscillations, we compare the RK-aligned OE-type DG method with the original OEDG method from \cite{O_Peng_2024} in discontinuous test problems. 
	
	\subsection{Superconvergence tests}
	The primary objective of this subsection is to investigate the superconvergence behavior for both 1D and 2D linear advection equations
	$
	u_t + \bm{\beta}\cdot \nabla u = 0,
	$ 
	subject to periodic boundary conditions on the domain $\Omega = [0,1]^d$. The focus is on two numerical errors $e_1$ and $e_2$ defined in \eqref{equ:error}.

	In these tests, the numerical solution $u_h$ is initialized by $u_h(x,0) = \tilde{P} u_0$, where $\tilde{P}$ is the projection operator defined in \Cref{def:1Dprojection} and \Cref{def:2Dprojection}.

	\begin{exmp}\label{test:1D}
		For the first test $\beta=1$, we consider the 1D linear advection equation with the initial condition $u_0(x) = \sin(2\pi x)$. This function is infinitely differentiable and therefore satisfies the regularity conditions necessary to observe the superconvergence phenomenon. Numerical errors at $t = 1.1$ are computed for both the RK-aligned OE-type DG method and the standard DG method based on $\mathbb{P}^k$ elements. The results, presented in Table~\ref{tab:OE_DG_ex1}, confirm the superconvergence properties established in \Cref{thm:superconvergence}: the RK-aligned OE-type DG method achieves the predicted $(k+2)$-order convergence, while the standard linear DG method converges at the known $(2k+1)$ rate.
	\end{exmp}

	\begin{table}[!htb]
		\centering
		\caption{Superconvergence of cell-average ($e_1$) and outflow-edge ($e_2$) errors for the 1D linear advection problem (Example \ref{test:1D}) using $\mathbb{P}^k$ elements.}
		\label{tab:OE_DG_ex1}
		\resizebox{\linewidth}{!}{%
			\begin{tabular}{c|ccccccccc}
				\toprule
				\multirow{2}{*}{$k$}
				& \multirow{2}{*}{$N_x$}
				& \multicolumn{4}{c}{RK-aligned OE-type DG}
				& \multicolumn{4}{c}{Standard RKDG}\\
				\cline{3-6} \cline{7-10}  
				& 
				& $e_1$ & rate
				& $e_2$ & rate
				& $e_1$ & rate
				& $e_2$ & rate \\ 
				\hline
				\multirow{8}{*}{1}
				& 16   & 2.95e-1 &-- & 3.03e-1 &-- & 2.92e-2 &-- & 3.00e-2 &-- \\
				& 32   & 8.85e-2 & 1.73 & 8.98e-2 & 1.75 & 4.00e-3 & 2.87 & 4.03e-3 & 2.90 \\
				& 64   & 1.31e-2 & 2.76 & 1.33e-2 & 2.76 & 5.10e-4 & 2.97 & 5.12e-4 & 2.98 \\
				& 128  & 2.06e-3 & 2.67 & 2.07e-3 & 2.68 & 6.41e-5 & 2.99 & 6.43e-5 & 2.99 \\
				& 256  & 2.85e-4 & 2.85 & 2.86e-4 & 2.85 & 8.02e-6 & 3.00 & 8.04e-6 & 3.00 \\
				& 512  & 3.72e-5 & 2.94 & 3.73e-5 & 2.94 & 1.00e-6 & 3.00 & 1.01e-6 & 3.00 \\
				& 1024 & 4.73e-6 & 2.97 & 4.74e-6 & 2.97 & 1.25e-7 & 3.00 & 1.26e-7 & 3.00 \\
				& 2048 & 5.97e-7 & 2.99 & 5.98e-7 & 2.99 & 1.57e-8 & 3.00 & 1.57e-8 & 3.00 \\
				\hline
				\multirow{6}{*}{2}
				& 16   & 3.28e-1  &--   & 3.35e-1  &--   & 1.92e-4   &--   & 1.98e-4   &--   \\
				& 32   & 6.99e-3  & 5.55 & 7.18e-3  & 5.55 & 6.25e-6   & 4.94 & 6.34e-6   & 4.96 \\
				& 64   & 2.09e-4  & 5.06 & 2.11e-4  & 5.09 & 1.98e-7   & 4.98 & 2.00e-7   & 4.99 \\
				& 128  & 9.93e-6  & 4.40 & 9.97e-6  & 4.40 & 6.20e-9   & 5.00 & 6.25e-9   & 5.00 \\
				& 256  & 5.54e-7  & 4.16 & 5.56e-7  & 4.17 & 1.94e-10  & 5.00 & 1.95e-10  & 5.00 \\
				& 512  & 3.28e-8  & 4.08 & 3.29e-8  & 4.08 & 6.03e-12  & 5.01 & 6.09e-12  & 5.00 \\
				\hline
				\multirow{4}{*}{3}
				& 16   & 2.31e-2  &--   & 2.36e-2  &--   & 6.17e-7   &--   & 6.10e-7   &--   \\
				& 32   & 1.16e-4  & 7.63 & 1.24e-4  & 7.57 & 4.97e-9   & 6.96 & 5.07e-9   & 6.91 \\
				& 64   & 3.15e-6  & 5.21 & 3.26e-6  & 5.25 & 3.92e-11  & 6.99 & 3.98e-11  & 6.99 \\
				& 128  & 1.04e-7  & 4.92 & 1.05e-7  & 4.96 & 3.09e-13  & 6.99 & 3.17e-13  & 6.97 \\
				\toprule[1.0pt]
			\end{tabular}%
		}
	\end{table}

 \begin{exmp}\label{test:1Dlinvar}
	We set the wave speed $\beta(x)=1+0.5\sin(2\pi x) > 0$, which meets the hypotheses of \Cref{sec:VC}. All other settings match \cref{test:1D}, including the boundary condition, final time, and initial data. \cref{tab:OE_DGlinvar_1D} presents the results and confirms the predicted $(k+2)$ superconvergence of the RK-aligned OEDG method proved in \cref{thm:superconvergence}, while the standard DG method converges at the known rate $(2k+1)$.
\end{exmp}

\begin{table}[!htb]
	\centering
	\caption{Superconvergence of cell-average ($e_1$) and outflow-edge ($e_2$) errors for the 1D variable-coefficient advection problem (\cref{test:1Dlinvar}) using $\mathbb{P}^k$ elements.}
	\label{tab:OE_DGlinvar_1D}
	\resizebox{\linewidth}{!}{%
		\begin{tabular}{c|ccccccccc}
			\toprule
			\multirow{2}{*}{$k$}
			& \multirow{2}{*}{$N_x$}
			& \multicolumn{4}{c}{RK-aligned OEDG}
			& \multicolumn{4}{c}{Standard RKDG}\\
			\cline{3-6} \cline{7-10}  
			& 
			& $e_1$ & rate
			& $e_2$ & rate
			& $e_1$ & rate
			& $e_2$ & rate \\ 
			\hline
			\multirow{8}{*}{1}
			& 64   & 3.77e-02 & -  & 3.78e-02 & -  & 8.33e-03 & -  & 8.35e-03 & -  \\
			& 128  & 6.30e-03 & 2.58  & 6.30e-03 & 2.58  & 1.14e-03 & 2.87  & 1.14e-03 & 2.87  \\
			& 256  & 9.28e-04 & 2.76  & 9.27e-04 & 2.76  & 1.45e-04 & 2.97  & 1.45e-04 & 2.98  \\
			& 512  & 1.21e-04 & 2.93  & 1.21e-04 & 2.94  & 1.82e-05 & 3.00  & 1.82e-05 & 3.00  \\
			& 1024 & 2.87e-05 & 2.98  & 2.87e-05 & 2.98  & 2.28e-06 & 3.00  & 2.28e-06 & 3.00  \\
			& 2048 & 1.92e-06 & 3.00  & 1.91e-06 & 3.00  & 2.85e-07 & 3.00  & 2.85e-07 & 3.00  \\
			& 4096 & 2.40e-07 & 3.00  & 2.39e-07 & 3.00  & 3.56e-08 & 3.00  & 3.56e-08 & 3.00  \\
			\hline
			\multirow{8}{*}{2}		& 64   & 1.82e-3 & -  & 1.83e-3 & -  & 6.47e-5 & -  & 6.52e-5 & -  \\
			& 128  & 1.12e-4 & 4.02  & 1.12e-4 & 4.02  & 2.11e-6 & 4.94  & 2.12e-6 & 4.94  \\
			& 256  & 6.54e-6 & 4.10  & 6.53e-6 & 4.10  & 6.68e-8 & 4.99  & 6.68e-8 & 4.99  \\
			& 512  & 3.91e-7 & 4.07  & 3.90e-7 & 4.07  & 2.09e-9 & 5.00  & 2.09e-9 & 5.00  \\
			& 1024 & 2.38e-8 & 4.03  & 2.38e-8 & 4.03  & 6.54e-11 & 5.00  & 6.54e-11 & 5.00  \\
			\hline
			\multirow{6}{*}{3}
			& 64   & 1.04e-4  & - & 1.05e-4  & - & 4.22e-7  & - & 4.27e-7  & - \\
			& 128  & 3.28e-6  & 4.99 & 3.29e-6  & 5.00 & 3.47e-9  & 6.93 & 3.48e-9  & 6.94 \\
			& 256  & 1.04e-7  & 4.98 & 1.04e-7  & 4.98 & 2.74e-11 & 6.98 & 2.74e-11 & 6.99 \\
			& 512  & 3.28e-9  & 4.99 & 3.28e-9  & 4.99 & \multicolumn{1}{c}{--} & \multicolumn{1}{c}{--} & \multicolumn{1}{c}{--} & \multicolumn{1}{c}{--} \\
			& 1024 & 1.03e-10 & 5.00 & 1.03e-10 & 5.00 & \multicolumn{1}{c}{--} & \multicolumn{1}{c}{--} & \multicolumn{1}{c}{--} & \multicolumn{1}{c}{--} \\
			& 2048 & 3.24e-12 & 4.99 & 3.24e-12 & 4.99 & \multicolumn{1}{c}{--} & \multicolumn{1}{c}{--} & \multicolumn{1}{c}{--} & \multicolumn{1}{c}{--} \\
			\hline
			
		\end{tabular}
	}
\end{table}

	
	For the 2D tests, we consider the linear advection equation on the domain $\Omega = [0,1]^2$ with the initial condition $u_0(x,y) = \sin\!\left(2\pi(x+y)\right)$. Numerical errors are evaluated at $t = 1$ to compare the RK-aligned OEDG method and the standard DG method using both $\mathbb{P}^k$ and $\mathbb{Q}^k$ elements. The results confirm the superconvergence properties established in \Cref{thm:superconvergence} for the multidimensional setting.
	
	\begin{exmp}\label{test:2DPk}
		For the $\mathbb{P}^k$-based elements, the computed errors and convergence rates for the RK-aligned OEDG and standard DG methods are reported in Table~\ref{tab:P_OEDG_P_DG}. The results show that both methods exhibit superconvergence behavior. The RK-aligned OEDG method achieves the predicted $(k+2)$-order accuracy, consistent with \Cref{thm:superconvergence}. For the standard DG method, we observed a convergence rate of $\min\{2k+1,\,k+3\}$ for both $e_1$ and $e_2$. Similar superconvergence rates have been reported in \cite{U_Cao_2025} for DG methods initialized with truncated Gauss--Radau interpolation, though a complete theoretical proof has not yet been established.
	\end{exmp}
	\begin{table}[!htb]
		\centering
		\caption{Errors and convergence rates of $\mathbb{P}^k$-based RK-aligned OE-type DG method and DG method for Example \ref{test:2DPk}.}
		\label{tab:P_OEDG_P_DG}
		\resizebox{\linewidth}{!}{%
			\begin{tabular}{c|ccccccccc}
				\toprule
				\multirow{2}{*}{$k$}
				& \multirow{2}{*}{$N_x\times N_y$}
				& \multicolumn{4}{c}{RK-aligned OE-type DG}
				& \multicolumn{4}{c}{Standard RKDG}\\
				\cline{3-6} \cline{7-10}  
				& 
				& $e_1$ & rate
				& $e_2$ & rate
				& $e_1$ & rate
				& $e_2$ & rate \\ 
				\hline
				\multirow{7}{*}{1}&	20×20         & 8.67e-02  &--    & 1.24e-01  &--     & 1.45e-02   &--        & 2.07e-02    &--        \\
				&40$\times$ 40         & 1.38e-02  & 2.65    & 1.97e-02  & 2.65     & 1.89e-03   & 2.94        & 2.68e-03    & 2.95        \\
				&80$\times$80         & 2.08e-03  & 2.73    & 2.96e-03  & 2.74     & 2.38e-04   & 2.99        & 3.38e-04    & 2.99        \\
				&160$\times$160       & 2.81e-04  & 2.89    & 3.97e-04  & 2.90     & 2.99e-05   & 3.00        & 4.24e-05    & 3.00        \\
				&320$\times$320       & 3.61e-05  & 2.96    & 5.12e-05  & 2.96     & 3.73e-06   & 3.00        & 5.30e-06    & 3.00        \\
				&640$\times$640       & 4.58e-06  & 2.98    & 6.49e-06  & 2.98     & 4.67e-07   & 3.00        & 6.63e-07    & 3.00        \\
				&1280$\times$1280     & 5.77e-07  & 2.99    & 8.17e-07  & 2.99     & 5.84e-08   & 3.00        & 8.28e-08    & 3.00        \\
				\hline
				\multirow{6}{*}{2} 
				& 20$\times$20         & 1.73e-02  &--       & 2.48e-02  &--        & 5.29e-05   &--          & 7.57e-05    &--          \\
				& 40$\times$40         & 9.22e-04  & 4.23    & 1.31e-03  & 4.24     & 1.68e-06   & 4.97       & 2.40e-06    & 4.98       \\
				& 80$\times$80         & 5.84e-05  & 3.98    & 8.28e-05  & 3.99     & 5.29e-08   & 4.99       & 7.55e-08    & 4.99       \\
				& 160$\times$160       & 3.65e-06  & 4.00    & 5.17e-06  & 4.00     & 1.65e-09   & 5.00       & 2.36e-09    & 5.00       \\
				& 320$\times$320       & 2.28e-07  & 4.00    & 3.23e-07  & 4.00     & 5.16e-11   & 5.00       & 7.35e-11    & 5.00       \\
				& 640$\times$640       & 1.42e-08  & 4.00    & 2.02e-08  & 4.00     & 1.60e-12   & 5.01       & 2.34e-12    & 4.97       \\
				\hline
				\multirow{5}{*}{3} 
				& 20$\times$20         & 1.23e-03  &--       & 1.76e-03  &--        & 1.51e-07   &--          & 2.14e-07    &--           \\
				& 40$\times$40         & 3.17e-05  & 5.28    & 4.52e-05  & 5.29     & 1.65e-09   & 6.51       & 2.34e-09    & 6.52         \\
				& 80$\times$80         & 9.06e-07  & 5.13    & 1.29e-06  & 5.13     & 2.01e-11   & 6.36       & 2.84e-11    & 6.36         \\
				& 160$\times$160       & 2.71e-08  & 5.07    & 3.84e-08  & 5.07     & 5.55e-13   & --       & 7.80e-13    &--         \\
				& 320$\times$320       & 8.28e-10  & 5.03    & 1.17e-09  & 5.03     & 3.19e-13   &--          & 4.52e-13    &--            \\
				\toprule[1.0pt]
			\end{tabular}
		}
	\end{table}

	\begin{exmp}\label{test:2DQk}
		For the $\mathbb{Q}^k$-based elements, the corresponding errors and convergence rates are reported in \Cref{tab:Q_OEDG_Q_DG}. The observed superconvergence orders are $k+2$ for the OEDG method and $2k+1$ for the standard DG method, thereby validating \cref{thm:superconvergence} even with $\mathbb{Q}^k$ elements for both $e_1$ and $e_2$.\end{exmp}
		\begin{table}[!htb]
			\centering
			\caption{Errors and convergence rates of $\mathbb{Q}^k$-based RK-aligned OEDG method and DG method for Example \ref{test:2DQk}.}
			\label{tab:Q_OEDG_Q_DG}
			\resizebox{\linewidth}{!}{%
				\begin{tabular}{c|ccccccccc}
					\toprule
					\multirow{2}{*}{$k$}
					& \multirow{2}{*}{$N_x\times N_y$}
					& \multicolumn{4}{c}{RK-aligned OEDG}
					& \multicolumn{4}{c}{Standard RKDG}\\
					\cline{3-6} \cline{7-10}  
					& 
					& $e_1$ & rate
					& $e_2$ & rate
					& $e_1$ & rate
					& $e_2$ & rate \\ 
					\hline
					\multirow{7}{*}{1} 
					& 20$\times$20         & 8.65e-02  &--       & 1.23e-01  &--        & 3.76e-03   &--          & 5.36e-03    &--          \\
					& 40$\times$40         & 1.29e-02  & 2.75    & 1.84e-02  & 2.75     & 4.76e-04   & 2.98       & 6.77e-04    & 2.98       \\
					& 80$\times$80         & 1.95e-03  & 2.73    & 2.76e-03  & 2.73     & 5.97e-05   & 3.00       & 8.48e-05    & 3.00       \\
					& 160$\times$160       & 2.63e-04  & 2.89    & 3.73e-04  & 2.89     & 7.47e-06   & 3.00       & 1.06e-05    & 3.00       \\
					& 320$\times$320       & 3.40e-05  & 2.95    & 4.81e-05  & 2.95     & 9.34e-07   & 3.00       & 1.33e-06    & 3.00       \\
					& 640$\times$640       & 4.31e-06  & 2.98    & 6.10e-06  & 2.98     & 1.17e-07   & 3.00       & 1.66e-07    & 3.00       \\
					& 1280$\times$1280     & 5.42e-07  & 2.99    & 7.68e-07  & 2.99     & 1.46e-08   & 3.00       & 2.07e-08    & 3.00      \\
					\hline
					\multirow{5}{*}{2} 	& 20$\times$20         & 3.73e-03   &--          & 5.37e-03    &--           & 3.84e-06   &--          & 5.53e-06    &--           \\
					& 40$\times$40         & 1.45e-04   & 4.69       & 2.06e-04    & 4.70         & 1.20e-07   & 5.00       & 1.73e-07    & 5.00         \\
					& 80$\times$80         & 7.19e-06   & 4.33       & 1.02e-05    & 4.34         & 3.75e-09   & 5.00       & 5.38e-09    & 5.00         \\
					& 160$\times$160       & 4.07e-07   & 4.14       & 5.77e-07    & 4.15         & 1.17e-10   & 5.00       & 1.68e-10    & 5.00         \\
					& 320$\times$320       & 2.43e-08   & 4.07       & 3.44e-08    & 4.07         & 3.52e-12   & 5.05       & 5.00e-12    & 5.07         \\
					\hline
					\multirow{5}{*}{3} & 20$\times$20         & 1.86e-04   &--          & 2.61e-04    &--           & 2.07e-09   &--          & 2.99e-09    &--           \\
					& 40$\times$40         & 1.64e-06   & 6.83       & 2.35e-06    & 6.80         & 1.64e-11   & 6.98       & 2.39e-11    & 6.96         \\
					& 80$\times$80         & 5.97e-08   & 4.78       & 8.48e-08    & 4.79         & 1.18e-13   & 7.11       & 1.77e-13    & 7.08         \\
					& 160$\times$160       & 1.95e-09   & 4.94       & 2.76e-09    & 4.94         &-   &--     &--    &--        \\
					& 320$\times$320       & 6.19e-11   & 4.98       & 8.76e-11    & 4.98         &--   &-  &--    &--        \\
					\hline
				\end{tabular}
			}
		\end{table}

	\subsection{Discontinuity-Capturing Tests: RK-aligned vs.~Original OEDG}
	\begin{exmp}\label{exp:1Dlinear}
This example confirms that RK-alignment preserves the non-oscillatory property of the original OEDG method. We consider the linear advection equation $u_t + u_x = 0$ with periodic boundary conditions and initial data $u_0(x) = \sin(2\pi x)$ for $x \in [0.3, 0.8]$ and $\cos(2\pi x)$ otherwise. The computation uses $N = 256$ uniform cells. As shown in \Cref{Fig:1DlinearP2} ($\mathbb{P}^2$ case), the RK-aligned solution is visually indistinguishable from the original OEDG result, confirming that the alignment procedure does not degrade solution quality.
	\end{exmp}
	\begin{figure}[!htbp]
		\centering
		\begin{subfigure}{0.48\linewidth}
			\includegraphics[width=\linewidth]{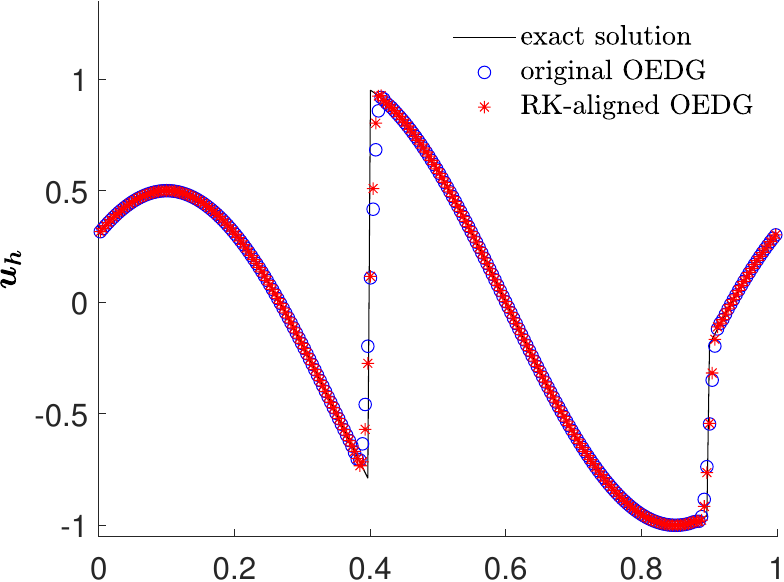}
			\caption{Results of $\mathbb{P}^2$-based original and RK-aligned OE-type DG at $t=1.1$.}
			\label{Fig:1DlinearP2}
		\end{subfigure}\hfill
		\begin{subfigure}{0.48\linewidth}
			\includegraphics[width=\linewidth]{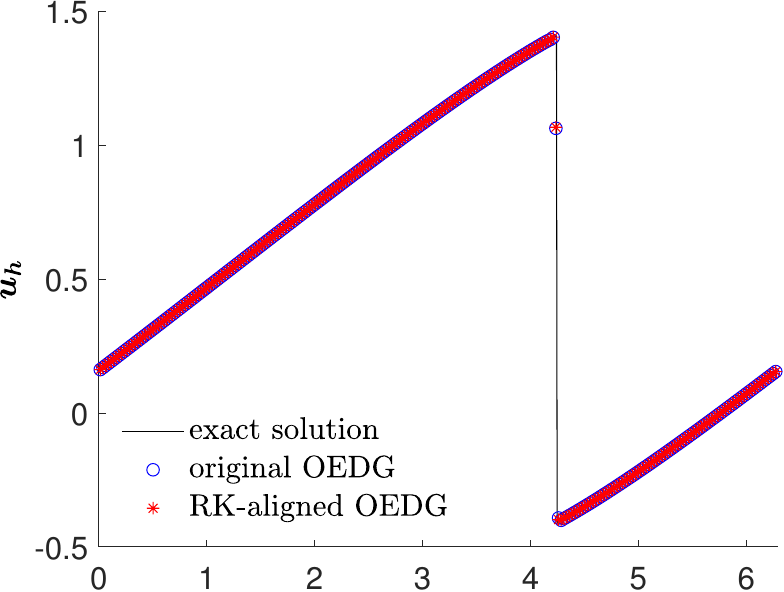}
			\caption{Comparison of $\mathbb{P}^3$-based original and RK-aligned OE-type DG at $t=2.2$.}
			\label{Fig:1DburgersP3}
		\end{subfigure}
		\caption{Results of OEDG solutions for \Cref{exp:1Dlinear} (left) and \Cref{exp:1Dburgers} (right).}
		\label{Fig:1D-combined}
	\end{figure}
	%
	%
	\begin{exmp}\label{exp:1Dburgers}
		We next consider the inviscid Burgers' equation 
		$
		u_t + \left(\tfrac{u^2}{2}\right)_x = 0
		$ 
		on the domain $\Omega = [0,2\pi]$ with periodic boundary conditions. The initial condition $u_0(x) = \sin(x) + 0.5$ is smooth but develops a discontinuity in finite time. \Cref{Fig:1DburgersP3} shows the numerical solutions at $t = 2.2$ obtained with the $\mathbb{P}^3$ RK-aligned and original OEDG methods using $N = 256$ uniform cells. Again, the RK-aligned OE-type DG solution is comparable to the original OEDG solution,  verifying that the alignment procedure does not degrade the resolution.
	\end{exmp}
	
	%


	\begin{exmp}[Woodward--Colella blast wave problem]
		This example simulates the interaction of two blast waves in the domain $[0,1]$ with reflective boundary conditions for the 1D Euler equation. The initial conditions are
		\[
		(\rho_0, v_0, p_0) =
		\begin{cases}
			(1, 0, 10^3),   & 0 < x < 0.1, \\
			(1, 0, 10^{-2}),& 0.1 < x < 0.9, \\
			(1, 0, 10^2),   & 0.9 < x < 1.
		\end{cases}
		\]
		\Cref{Fig:1DEulerP2} shows the numerical results of densities and pressures at $t = 0.038$ obtained with the third-order RK-aligned and original OEDG methods on a uniform mesh of 640 cells. The DG solution polynomials are plotted, while the reference solution is computed using the Lax--Friedrichs scheme with 300{,}000 uniform cells. The results demonstrate that the proposed RK-aligned OEDG method is as effective as the original OEDG method in suppressing oscillations and in accurately resolving the complex wave structure.
	\end{exmp}

	\begin{figure}[!htbp]
		\centering
		
		\begin{subfigure}{0.48\linewidth}
			\includegraphics[width=1\linewidth]{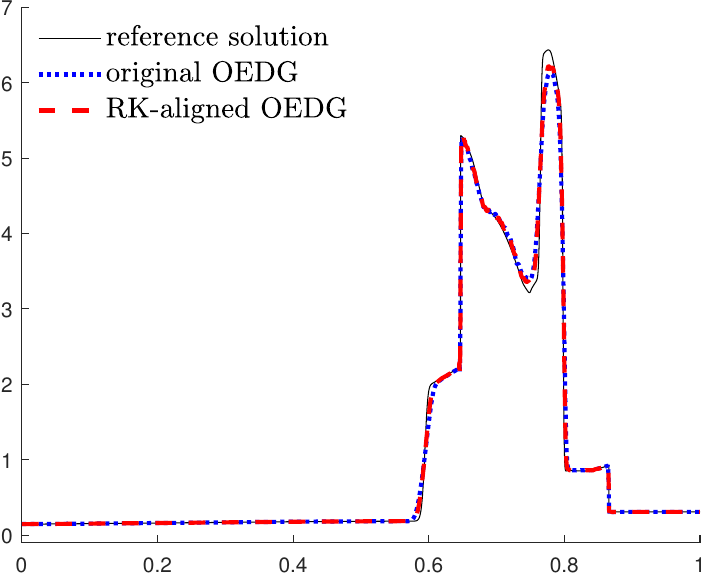}
			\label{Density}
		\end{subfigure}
		\hfill 
		\begin{subfigure}{0.48\linewidth}
			\includegraphics[width=1\linewidth]{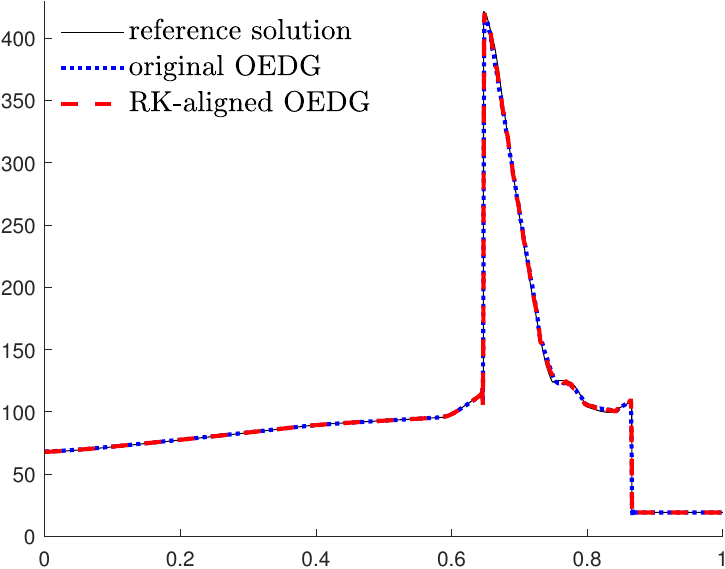}
			\label{Pressure}
		\end{subfigure}
		\caption{Density (left) and pressure (right) at $t = 0.038$ obtained with $\mathbb{P}^2$-based original and RK-aligned OEDG methods. The DG solution polynomials are shown. }
		\label{Fig:1DEulerP2}
	\end{figure}
	
	\begin{exmp}
		We solve a 2D Riemann problem for the compressible Euler equations on the domain $[0,1]^2$ with outflow boundary conditions using $320 \times 320$ uniform cells. The $\mathbb{P}^2$ RK-aligned and original OEDG methods are employed. The initial data are
		\[
		(\rho_0, \mathbf{v}_0, p_0) =
		\begin{cases}
			(0.8,\,0,\,0,\,1),      & x<0.5,\, y<0.5, \\
			(1,\,0.7276,\,0,\,1),   & x<0.5,\, y>0.5, \\
			(1,\,0,\,0.7276,\,1),   & x>0.5,\, y<0.5, \\
			(0.5313,\,0,\,0,\,0.4), & x>0.5,\, y>0.5,
		\end{cases}
		\]
		yielding two stationary contact discontinuities and two shocks. As shown in \Cref{Fig:2DeulerP2}, the original and RK-aligned OEDG results at $t=0.25$ are virtually identical and oscillation-free.
	\end{exmp}
	
	\begin{figure}[!htbp]
		\centering
		
		\begin{subfigure}{0.48\linewidth}
			\includegraphics[width=1\linewidth]{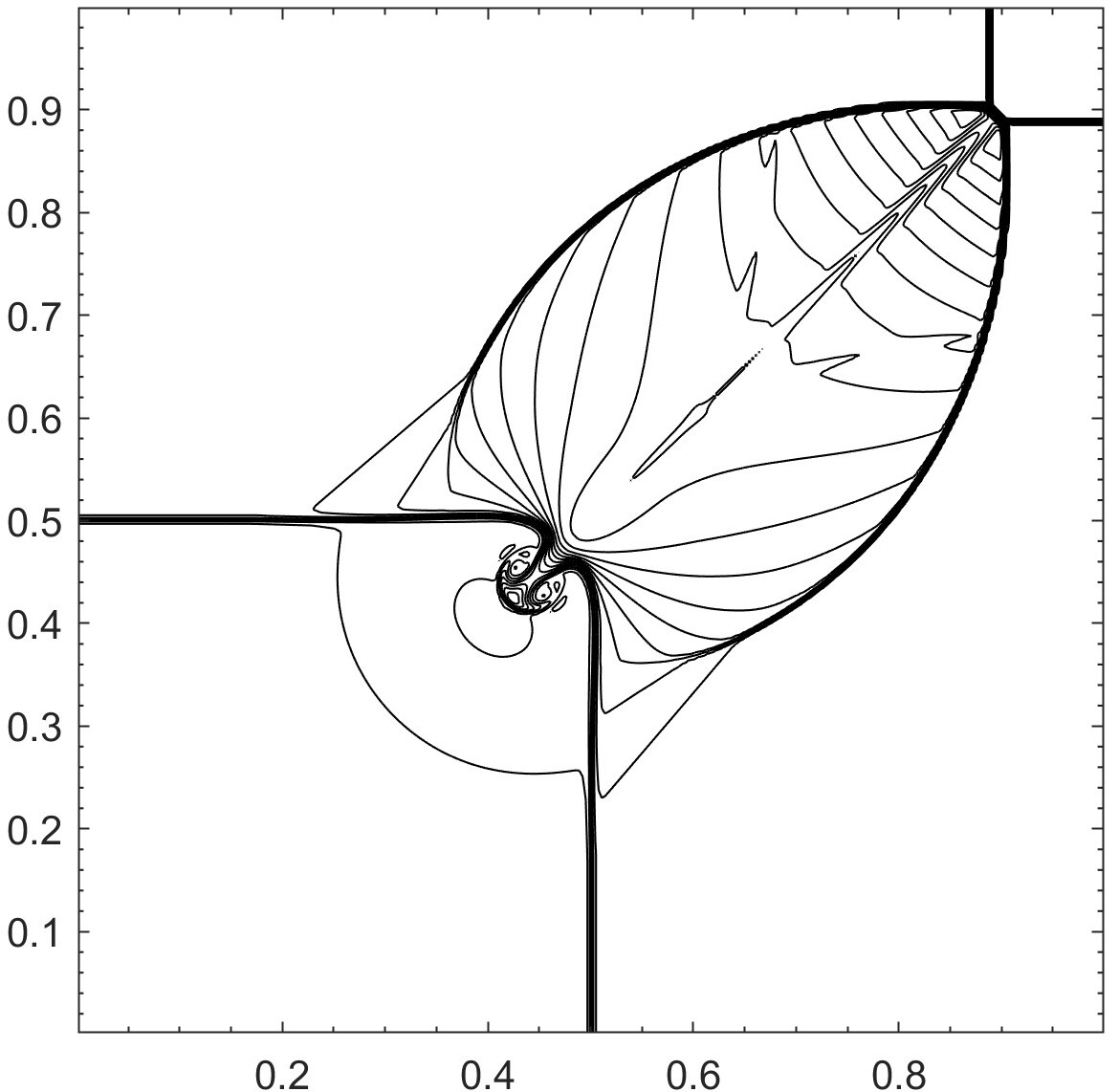}
		\end{subfigure}
		\hfill 
		\begin{subfigure}{0.48\linewidth}
			\includegraphics[width=1\linewidth]{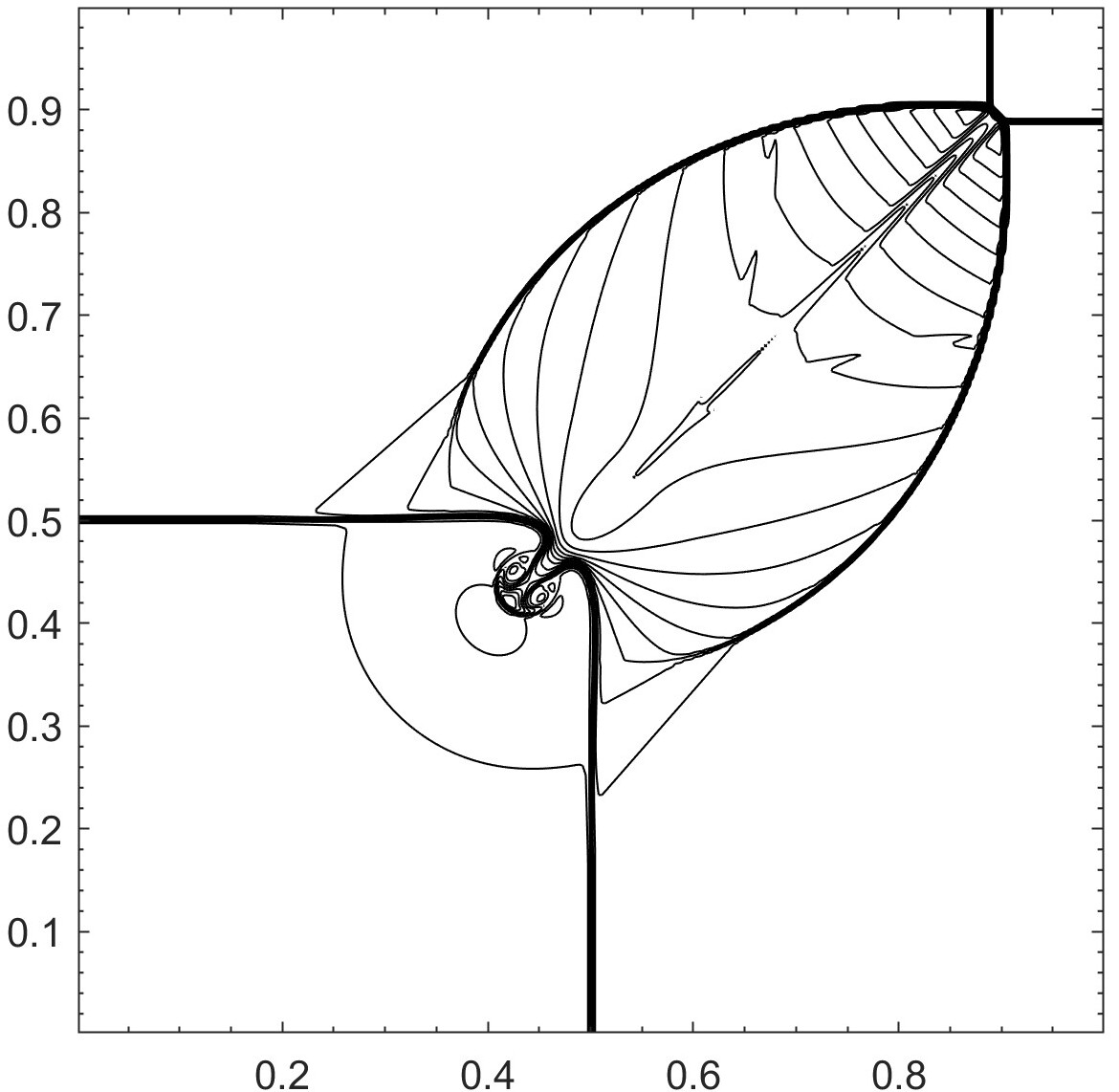}
		\end{subfigure}
		\caption{Density for $\mathbb{P}^2$-based original (left) and RK-aligned (right) OEDG solutions.}
		\label{Fig:2DeulerP2}
	\end{figure}

	In summary, the experiments confirm (i) the predicted $(k+2)$ superconvergence for smooth solutions (cell averages and outflow-edge averages), and (ii) that RK alignment preserves the shock-capturing robustness of the original OEDG stabilization. This demonstrates that the structural modification that enables our rigorous superconvergence proof comes at virtually no cost to the scheme's excellent practical robustness and shock-capturing capabilities. 
	We note that all tests respected the standard CFL condition, and we did not observe any stability issues. 
	The RK-aligned OEDG method thus represents a clear theoretical and practical advancement, providing the first rigorous superconvergence guarantee for a highly effective, parameter-free, oscillation-eliminating DG framework.

\section{Concluding remarks}\label{sec:conclusion}
In this work, we have established a rigorous superconvergence theory for oscillation-eliminating discontinuous Galerkin (OEDG) methods by introducing a structural synchronization between the nonlinear dissipation step and the Runge--Kutta stages. We proved that the proposed RK-aligned OE-type DG scheme achieves $(k{+}2)$-th order superconvergence to a tailored projection of the exact solution for linear hyperbolic conservation laws in both one and two dimensions. To the best of our knowledge, this constitutes \emph{the first rigorous superconvergence result for a fully discrete, nonlinear DG scheme capable of robust shock-capturing}. This theoretical milestone bridges the gap identified in our prior work \cite{O_Peng_2024}, elevating the analytical understanding of OEDG from optimal $(k{+}1)$-th order convergence to the higher-order behavior observed in practice.

Crucially, the theoretical modification required for this analysis (RK alignment) does not compromise the practical utility of the method. The scheme retains the desirable properties of the original OEDG framework, including local conservation, scale invariance, and parameter-free oscillation suppression. Numerical experiments confirm that the RK-aligned variant matches the shock-capturing robustness of the original method for discontinuous problems while delivering the predicted superconvergence gains for smooth solutions.

The analytical framework developed herein opens several promising avenues for future research. The most immediate extension is to nonlinear conservation laws; the machinery of stage-aligned correction functions and structure-preserving projections provides a solid foundation for tackling the complexities of nonlinear fluxes. A more challenging but equally important direction is the extension to unstructured meshes (e.g., triangular elements). This remains an open problem, as the construction of multi-dimensional correction functions for $P^k$ elements on general geometries requires non-trivial innovations beyond the tensor-product structures exploited here.

Overall, this work strengthens the mathematical foundations of nonlinear high-order schemes by reconciling rigorous error analysis with robust shock-capturing. We hope this framework paves the way for the design of future structure-preserving, non-oscillatory methods with provable high-order properties.

\begin{changemargin}{-0.2cm}{-0.2cm}  %
		
		\renewcommand\baselinestretch{0.86}
		\bibliography{refs}
		\bibliographystyle{siamplain}
		
	\end{changemargin}

\end{document}